\documentclass[11pt]{amsart}
\usepackage{amsmath,amsthm,amscd,amssymb, color}
\usepackage{thmtools}
\usepackage[colorlinks, citecolor = blue,backref]{hyperref}
\usepackage[alphabetic]{amsrefs}
\usepackage[capitalize, nameinlink]{cleveref}
\usepackage{enumerate}
\usepackage[margin=1.333in]{geometry}
\usepackage{graphicx}
\usepackage[font=footnotesize]{caption}

\usepackage{array}   
\newcolumntype{L}{>{$}l<{$}} 
\newcolumntype{R}{>{$}r<{$}} 
\newcolumntype{C}{>{$}c<{$}}

\newcommand{\Q}{\mathbb Q}

\newcommand{\Z}{\mathbb Z}
\newcommand{\N}{\mathbb N}

\renewcommand{\phi}{\varphi}
\newcommand{\eps}{\varepsilon}

\newcommand{\bmx}{\left( \begin{matrix}}
\newcommand{\emx}{\end{matrix} \right)}

\newcommand{\new}{\mathrm{new}}

\renewcommand{\mod}{\bmod}
\newcommand{\sqf}{\mathrm{sqf}}

\newcommand{\bigdot}{\boldsymbol{\cdot}}

\usepackage{bbm}

\newcommand{\leg}{\overwithdelims ()}

\DeclareMathOperator{\PGL}{PGL}
\DeclareMathOperator{\tr}{tr} 
\DeclareMathOperator{\lcm}{lcm}

\newtheorem{lem}{Lemma}
\numberwithin{lem}{section}
\newtheorem{prop}[lem]{Proposition}
\newtheorem{thm}[lem]{Theorem}
\crefname{thm}{Theorem}{Theorems}
\newtheorem{cor}[lem]{Corollary}
\newtheorem{conj}[lem]{Conjecture}
\crefname{conj}{Conjecture}{Conjectures}

\crefname{question}{Question}{Questions}

\theoremstyle{remark}
\newtheorem{rem}[lem]{Remark}

\theoremstyle{definition}

\numberwithin{equation}{section}

\begin{document}

\title{Distribution of local signs of modular forms and murmurations of Fourier coefficients}
\author{Kimball Martin}
\address{Department of Mathematics, University of Oklahoma, Norman, OK 73019 USA}
\email{kimball.martin@ou.edu}

\address{Department of Mathematics $\cdot$ International Research and Education Center, Graduate School of Science, Osaka Metropolitan University, Osaka 558-8585, Japan}
\email{kimball@omu.ac.jp}


\maketitle

Corrected version (\today)\footnote{After this paper appeared in \emph{Mathematika}\ (2025), I discovered minor mathematical misprints in Section 4.2.  This version corrects those errors.  (See footnote in Section 4.2 for changes.)}

\begin{abstract}
Recently, we showed that global root numbers of modular forms are biased toward $+1$.  Together with Pharis, we also showed an initial bias of Fourier coefficients towards the sign of the root number.  First, we prove analogous results with respect to local root numbers.

Second, a subtle correlation between Fourier coefficients and global root numbers, termed murmurations, was recently discovered for elliptic curves and modular forms.  We conjecture murmurations in a more general context of different (possibly empty) combinations of local root numbers.

Last, an appendix corrects a sign error in our joint paper with Pharis.
\end{abstract}


\section{Introduction}


Here we study the traces of Atkin--Lehner operators on spaces of newforms $S_k^\new(N) = S_k^\new(\Gamma_0(N))$.  There are two main reasons we are interested in this: (1) to understand distributions of local root numbers of newforms; (2) to explore correlation of Fourier coefficients of newforms with respect to local root numbers, and in particular explore variations on recently discovered murmuration phenomena.  These questions are local analogues of recent discoveries about global root numbers.

\subsection{Distributions of local root numbers}

In \cite{me:refdim,me:rootno}, we observed a bias of newforms towards global root number $+1$, even though asymptotically these account for 50\% of newforms.  Namely, for fixed pair $(k, N)$, outside of a prescribed set of exceptions, there are strictly more newforms in $S^\new_k(N)$ with root number $+1$ than $-1$.  Moreover, the excess number of forms with root number $+1$ is essentially independent of $k$, and is typically an elementary factor times the class number of $\Q(\sqrt{-N})$.

In fact \cite{me:refdim} was primarily concerned with the distribution of local root numbers in $S_k^\new(N)$.  Suppose $N$ is squarefree, and $q_1, \dots, q_m$ are primes dividing $N$.  We obtained a criterion for when the local root numbers (i.e., Atkin--Lehner eigenvalues) at $q_1, \dots, q_m$  are perfectly equidistributed in $S_k^\new(N)$, i.e., the number of newforms with prescribed local signs at $q_1, \dots, q_m$ does not depend on the choice of signs.  We also showed that there is a bias towards/away all local signs being $-1$, with the direction of the bias depending on the parities of $\frac k2$ and the number of prime divisors of $N$.

The motivation for studying the distribution of local root numbers in \cite{me:refdim} was for applications to congruences mod 2.  Suppose further that $N$ is a squarefree product of an odd number of primes.  In \cite{me:cong2}, we showed that, apart from levels of the form $N = 2p_1 p_2$ when $k=2$, if the local Atkin--Lehner signs are perfectly equidistributed for $q_1, \dots, q_m$, then for any newform $f \in S_k(N)$ and any prescribed choice of local signs at $q_1, \dots, q_m$, there is a newform $g \in S_k(N)$ with those prescribed signs which is congruent to $f$ mod $2$.  Further when $k=2$, where there is a bias towards all local signs $-1$, one can do the same (without perfect equidistribution) when one prescribes the local signs of $g$ to all be $-1$. 

The approach in \cite{me:refdim,me:rootno} is via explicit trace formulas.  However for general levels $N$ and arbitrary collections of prime divisors $q_1, \dots, q_m$ of $N$, the trace formula is rather complicated, making a complete generalization of \cite{me:refdim} difficult.  Our first goal here is to study the distribution of the local root number at a single prime $q$ for general levels $N$.  It would be interesting to see if there are similar applications to mod 2 congruences as in \cite{me:cong2} for general $N$, but we do not pursue this here.  (However, see \cref{prop:quadtwist} for when quadratic twisting implies an analogous mod 2 congruence result.)

Let $q$ be a prime, $r \ge 1$, and $M \ge 1$ such that $q \nmid M$.  Denote by  
$S_k^\new(q^rM)^{\pm_q}$ the subspace of $S_k^\new(q^r M)$ which is the $\pm$-eigenspace of the Atkin--Lehner operator $W_q$ at $q$.
Define
\[ \Delta_k(q^r, M) =  \dim S_k^\new(q^rM)^{+_q} - \dim S_k^\new(q^rM)^{-_q} .\]
In other words, $\Delta_k(q^r, M)$ is the trace of $W_q$ on $S_k^\new(q^r M)$.
In \cref{sec:dim} we obtain very explicit formulas for $\Delta_k(q^r, M)$ which are of the form
\begin{equation} \label{eq:Deltak-form}
\Delta_k(q^r, M) = 
\begin{cases}
C_1 h_{\Q(\sqrt{-q})} + \delta_{r=1}  D_1 & \text{if } r \text{ is odd}, \\
C_2 +  \delta_{r=2} (k-1) D_2  & \text{if } r \text{ is even}.
\end{cases}
\end{equation}
Here $\delta_*$ is the Kronecker delta, and $C_i$ and $D_i$ are elementary functions which depend on $q$ and $r$, 
depend on the prime factorization of $M$ (i.e., are expressible in terms of multiplicative functions of $M$), and only depend in a mild way on $k$.  In general, these functions depend on $k \mod 24$, and whether $k=2$, but in most cases only involve a factor of $(-1)^\frac k2$.
The explicit form of these functions breaks up into various cases, but for instance when $q \equiv 1 \mod 4$, $r=1$, and $M$ is odd, we have
\[ \Delta_k(q,M) = \frac 12 (-1)^{\frac k2} \kappa_{-q}(M) h_{\Q(\sqrt{-q})}  + \delta_{k=2} \mu(M), \]
where $\mu$ is the M\"obius function, and $\kappa_{-q}$ is the multiplicative function defined by \eqref{eq:kappa-def}.

As a consequence, outside of certain exceptional cases, we characterize when $\Delta_k(q^r,M) = 0$, and specify the sign of $\Delta_k(q^r, M)$ when it is nonzero.  Let $\omega(M)$ be the number of primes dividing $M$, $\omega_1(M)$ the number of primes sharply dividing $M$, and $\omega_2(n; M)$ the number of $p^2 \parallel M$ such that ${n \leg p} = 1$.

\begin{thm} [odd exponent] \label{thm11}
Let $q$ be a prime, $M \ge 1$ be coprime to $q$ and write $M = 2^e M'$ where $M'$ is odd.  Let $r \ge 1$ be an odd integer.  If $r \le 3$ assume $q \ge 5$,
and if $r=1$ further assume that $k \ge 4$ or $M$ is not squarefree.  

Then $\Delta_k(q^r, M) = 0$
if and only if 
(i) $M'$ is not cubefree; (ii) ${-q \leg p} = 1$ for some $p \parallel M'$; (iii) $e \ge 5$;
(iv) $e = 4$ and $q \equiv 1 \mod 4$; or
(v) $e = 1, 2, 3$ and $q \equiv 7 \mod 8$.

Moreover, when $\Delta_k(q^r, M) \ne 0$, its sign is $(-1)^{k/2 + \omega_1(M') + \omega_2(-q, M')} b_{q,e}$, where $b_{q,e}$ is the sign of the quantity $\alpha_1(-q; e)$ in \cref{tab:alpha1}.  I.e., we may take $b_{q,e} = +1$ if $e=0$; $-1$ if $e=1,2$; ${-1 \leg q}$ if $e=3$; and $-{2 \leg q}$ if $e=4$.  Hence this sign is simply $(-1)^{k/2 + \omega(M)}$ when $M$ is squarefree.
\end{thm}

\begin{thm}  [even exponent $\ge 4$] \label{thm12}
Let $q$ be a prime, $r \ge 4$ even, and $M  \ge 1$ coprime to $q$.  Assume $q^r \ne 16$ and write $M = 2^e M'$ where $M'$ is odd. 

Then $\Delta_k(q^r, M) = 0$ if and only if (i) $M'$ is not cubefree, (ii) $16 \mid M$ or (iii) $p \equiv 1 \mod 4$ for some $p \parallel M'$.

If $\Delta_k(q^r, M) \ne 0$, then its sign is $(-1)^{\frac k2 + \omega_1(M') + \omega_2(-1;M')} b_{2,e}$, where $b_{2,e} = 1$ if $e = 0, 3$ and $b_{2,e} = -1$ if $e = 1, 2$.  In particular, this sign is $(-1)^{\frac k2 + \omega(M)}$ if $M$ is squarefree.
\end{thm}

See \cref{sec:dim} for analogous results for other cases (e.g., when $q^r$ is small, or when $r=1$, $k=2$ and $M$ is squarefree).  We remark that sometimes $\Delta_k(q^r, M) = 0$ is forced upon us by the action of quadratic twists---see \cref{prop:quadtwist}---but quadratic twisting does not suffice to explain most cases of perfect equidistribution of local root numbers.

Due to the $\delta_{r=2} (k-1) D_2$ term in \eqref{eq:Deltak-form}, the behavior is different when $r=2$.  In this case we describe the asymptotic behavior, which in the following setting asserts a bias towards local root number $-1$.

\begin{prop} [exponent 2] \label{prop13} 
Fix a prime $q$, and consider $k + M \to \infty$ such that $k \ge 2$ is even and $M \ge 1$ is coprime to $q$.  Then $\Delta_k(q^2, M) \to -\infty$.  More precisely $\Delta_k(q^2, M) \sim \frac{1-k}{12} \kappa_\infty(M)$, where $\kappa_\infty$ is the multiplicative function defined by \eqref{eq:kapinfty}.
\end{prop}

We also establish the asymptotic behavior in $q$ under local conditions on $M$ (see \cref{prop:q2}).  The fact that local root number distributions behave differently in $r=2$ parallels the fact that the bias of global root numbers is different for levels which are perfect squares (see \cite{me:rootno}).  We do not have a compelling intuitive explanation for why this is (for local or global root numbers), but we do note that, when $q$ is odd, $r=2$ is precisely the case where the class of possible local representations $\pi_q$ at $q$ associated to newforms $f \in S_k^\new(q^rN)$ includes ramified principal series and ramified twists of Steinberg representations, all of which have local root number ${-1 \leg q}$.  However, such forms cannot account for the bias towards local root number $-1$ when $q \equiv 1 \mod 4$.

We also remark that the fact that \eqref{eq:Deltak-form} only depends in a mild way on the weight implies the following.

\begin{cor} [boundedness in $k$] Fix a prime $q$ and $r \ge 1$.  Assume $r \ne 2$.  Then $\lvert \Delta_k(q^r, M) \rvert$ is bounded as $k \to \infty$.
\end{cor}

This boundedness is also a simple consequence of existing trace formulas, but perhaps was not explicitly stated in the literature.  In fact our formulas yield that $\lvert \Delta_k(q^r, M) \rvert$ is typically constant in $k$.

\begin{rem} The corollary implies that, if $r \ne 2$, the trace of $W_q$ on the $q$-new part of $S_k(q^r M)$ is bounded in $k$, and in fact only depends on $k$ a mild way.  One can view this as very strict equidistribution of the Atkin--Lehner sign at $q$ in the weight aspect.  A more refined problem is to study the distribution of $q$-adic Galois representations for modular forms at $q$; see recent work of Bergdall and Pollack \cite{BP} taking $r=1$. 
\end{rem}

\subsection{Correlation of initial Fourier coefficients with local signs}

In \cite{me:pharis}, we showed that for squarefree levels $N$, the trace of a Hecke operator $T_\ell$ on the subspace of forms in $S_k^\new(N)$ with root number $+1$ (resp., $-1$) is positive (resp., negative) for $\ell$ small relative to $N$.\footnote{There is a sign error for this result in \cite{me:pharis} when $k \equiv 0 \mod 4$.  We correct this in \cref{appendix}.} In other words, for small $\ell$, the sign of the Fourier coefficients $a_\ell(f)$ are biased towards the sign of the root number of $f$.  Combined with \cite{me:refdim}, this means that for small $\ell$ the Fourier coefficients $a_\ell(f)$ have a positive (resp., negative) bias for forms with the more common (resp., less common) global root number.

Here we obtain analogous results for local root numbers: for small $\ell$, the
Fourier coefficients $a_\ell(f)$ have a positive (resp., negative) bias for forms with the more common (resp., less common) local root number.  However, in this case the bias only occurs under suitable congruence/divisibility conditions, and when these are not satisfied, there is essentially no bias for the $a_\ell$'s based on the local root number.

For simplicity, we restrict to levels of the form $N= qM$ where $M$ is squarefree or twice a squarefree number.   

\begin{thm} [bias of initial Fourier coefficients] \label{thm16}
Let $q, \ell$ denote primes such that $\ell < \frac q4$.  Let $M$ be a squarefree or twice a squarefree number which is coprime to $q\ell$.  If $k=2$, further assume $4 \mid M$.
Suppose  $\Delta_k(q,M) \ne 0$.   
\begin{enumerate}
\item Either $\tr_{S_k^\new(q M)} T_\ell W_q$ is 0 or it has the same sign as $\Delta_k(q,M)$.  It is 0 if and only if (i) ${-q \ell \leg p} = 1$ for some odd $p \mid M$ or (ii) $M$ is even and $q\ell \equiv 7 \mod 8$.

\item
If $q$ is sufficiently large with respect to $k, M, \ell$, and if $\tr_{S_k^\new(q M)} T_\ell W_q \ne 0$, then the trace of $T_\ell$ on $S_k^\new(q M)^{\pm_q}$ has the same sign as $\pm \Delta_k(q,M)$.
\end{enumerate}
\end{thm}

We remark that if the odd part of $M$ is not squarefree, then the correlation between the signs of $\Delta_k(q,M)$ and $\tr_{S_k^\new(q M)} T_\ell W_q$ will alternate depending on quadratic residue symbols involving odd primes with $p^2 \parallel M$.  
 On the other hand, if $p^3 \mid M$ for an odd $p$ or $32 \mid M$, then one has perfect equidistribution of both local root numbers at $q$ and trace of $T_\ell$ with respect to the local root number at $q$.  See \cref{sec:trlsmall} for details.

Moreover, it is clear from our trace formulas that one can similarly treat levels of the form $N =q^r M$ with $r \ge 2$, and the behavior that occurs is similar to the case of $r=1$.

\subsection{Murmurations} \label{sec:intro-murm}
Recently \cite{HLOP} numerically discovered an oscillatory pattern, which they call murmurations, in averages of $a_\ell$'s (for $\ell$ prime) over elliptic curves of fixed rank or root number.  Sutherland\footnote{See: \url{https://math.mit.edu/~drew/murmurations.html}} did further extensive calculations---for elliptic curves, modular forms, and abelian surfaces---to help clarify and make precise the murmuration phenomena.
The calculations for modular forms rely on the trace formula for $\tr_{S_k^\new(N)} T_\ell W_N$.  (One does not need to restrict to prime $\ell$, but we will for simplicity.)
We describe the phenomenon for modular forms, and then propose some generalizations, both for modular forms and elliptic curves.

Fix $k \ge 2$ even and let $\mathcal F = \mathcal F_k$ be a suitably large family of weight $k$ newforms, say all weight $k$ newforms (with trivial nebentypus), or all of those with squarefree level.  Let $\mathcal F(X,Y)$ (resp.\ $\mathcal F^{\pm}(X,Y)$) be the set of newforms $f \in \mathcal F$ with level $X \le N \le Y$ (resp.\ and have have root number $\pm 1$).  Denote by $\mathcal F(X,Y)^{(\ell)}$ (resp.\ $\mathcal F^{\pm}(X,Y)^{(\ell)}$ be the subset of such forms with level $N$ coprime to $\ell$.  The murmuration phenomena is the numerical observation that, for a fixed $\beta > 1$, the averages
\[ A_{\mathcal F}^\pm(\ell,X; \beta) = \frac 1{\#\mathcal F^{\pm}(X, \beta X)^{(\ell)}} \sum_{f \in \mathcal F^{\pm}(X,\beta X)^{(\ell)}} \ell^{1-\frac k2} a_\ell(f) \]
tend to continuous \emph{murmuration functions} $M^\pm_{\mathcal F}(x; \beta)$ as $\ell, X \to \infty$ such that $\frac \ell X \to x$.  The restriction to forms in  $\mathcal F^{\pm}(X,Y)^{(\ell)}$, as opposed to $\mathcal F^{\pm}(X,Y)$, is just a computational convenience and will not affect asymptotics (see \cref{sec:lmidN}).
The normalization factor $\ell^{1-\frac k2}$ weights the Fourier coefficients so that each term is $O(\sqrt \ell)$.  In all numerical calculations, we follow Sutherland and take $\beta = 2$.

\begin{figure}
\begin{minipage}{.5\textwidth}
\captionof{figure}{Weight 2 murmurations for $W_N$ for squarefree levels $1000 \le N \le 2000$}
 \includegraphics[width=\textwidth]{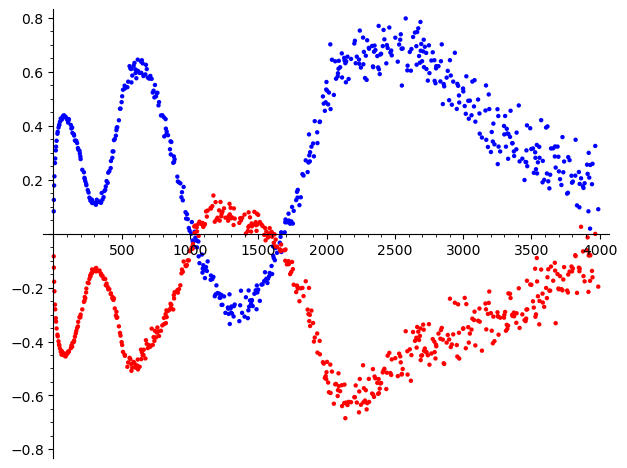}
\label{fig:fixedroot1}
\end{minipage}%
\begin{minipage}{.5\textwidth}
\captionof{figure}{Weight 2 murmurations for $W_N$ for squarefree levels $2000 \le N \le 4000$}
 \includegraphics[width=\textwidth]{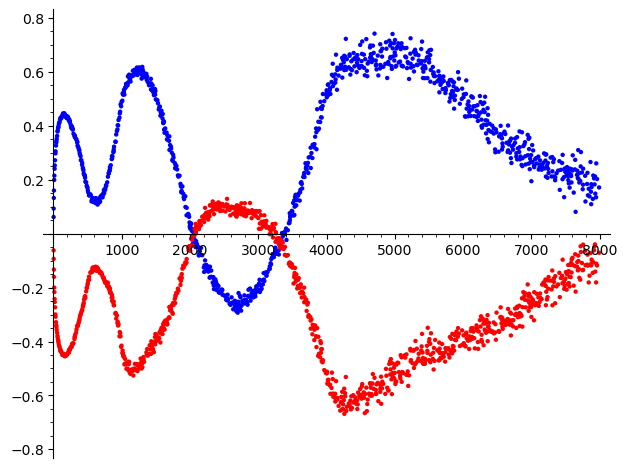}
\label{fig:fixedroot2}
\end{minipage}
\end{figure}

See \cref{fig:fixedroot1,fig:fixedroot2} for plots of $A_{\mathcal F}^\pm(\ell,X; 2)$ for $X = 1000$ and $X = 2000$ for the weight 2 squarefree level family.  Blue dots represent root number $+1$ and red dots $-1$.
We labeled the horizontal axes by $\ell$, but really one should think of the horizontal axes as being labeled by $\frac \ell X$, so both graphs have horizontal range $0 < \frac \ell X < 4$.  That the graphs have a limit in this scale is called \emph{scale invariance} in $\frac \ell X$.  The limiting murmuration functions $M^\pm_{\mathcal F}$ oscillate infinitely \cite{zubrilina}.

To analyze these murmurations, since $M^+_{\mathcal F}(x; \beta) = -M^-_{\mathcal F}(x; \beta)$ (see \cref{cor:63}) one can instead study a single weighted sum
\begin{equation} \label{eq:drew-avgs}
 A_{\mathcal F}(\ell,X; \beta) = \frac 1{\#\mathcal F(X,\beta X)^{(\ell)}}\sum_{f \in \mathcal F(X,\beta X)^{(\ell)}} w(f) \ell^{1 - \frac k2} a_\ell(f), 
\end{equation}
where $w(f)$ is the root number of $f$ and $\mathcal F(X,Y) =
\mathcal F^{+}(X,Y) \cup \mathcal F^{-}(X,Y)$.  
Then the assertion is that $A_{\mathcal F}(\ell,X; \beta) \to M_{\mathcal F}(x; \beta)$ as $\frac \ell X \to x$ where $M_{\mathcal F} = \frac 12( M^+_{\mathcal F} - M^-_{\mathcal F}) = M^+_{\mathcal F}$.

Zubrilina \cite{zubrilina} proved such murmurations for $A_{\mathcal F}$ when $\mathcal F$ is the family of weight $k$ newforms of squarefree level.  In fact, Zubrilina proves a localized version of murmurations---essentially this means one can work with intervals of the form $[X, \beta X]$ where $\beta \to 0$ as $X \to \infty$.  This leads to a continuous \emph{murmuration density function}, and one obtains the murmuration functions $M_{\mathcal F}(x; \beta)$  by integrating the murmuration density function (and this implies continuity in $\beta$).

As trace formulas for more general Atkin--Lehner operators times Hecke operators, i.e., $\tr_{S_k^\new(N)} T_\ell W_Q$ (where $Q = Q_N \mid N$), are in some ways quite similar to $\tr_{S_k^\new(N)} T_\ell W_N$, one might wonder if murmurations similarly exist with respect to Atkin--Lehner eigenvalues.  Here one needs to choose how to vary $Q$ along with $N$ in these the averages.

One possibility is to consider averages of the form
\begin{equation} \label{eq:gen-avgs}
A^{\mathcal Q}_{\mathcal F}(\ell,X; \beta) = \frac 1{\#\mathcal F(X,\beta X)^{(\ell)}} \sideset{}{'}\sum_{X \le N \le \beta X} \sum_{f \in \mathcal F(N)} w_Q(f) \sqrt{\frac NQ}  \ell^{1 - \frac k2} a_\ell(f),
\end{equation}
where $\mathcal Q$ is a sequence of divisors $Q \mid N$ for each level $N$ appearing the family $\mathcal F$.  Here $\mathcal F(N) = \mathcal F(N,N)$, and $w_Q(f)$ is the $W_Q$-eigenvalue of $f$.  The notation $\sum^\prime$ means we restrict to summing $N$ coprime to $\ell$.  The normalization factor $\sqrt{\frac N Q}$ is included to keep the averages at about the same size so they do not tend to $0$ as $X \to \infty$ (see \cref{sec:murmII}).

Note that if each $Q = N$ then \eqref{eq:gen-avgs} becomes \eqref{eq:drew-avgs}.  
At the other extreme if each $Q = 1$ then we are considering sums without any root numbers (as $w_1(f) = 1$).
See \cref{fig:murmN-1a,fig:murmN-1b} for plots of the averages in \eqref{eq:gen-avgs} in this case when $k=2$.  (Graphs for $k=4$ are roughly similar.)  Without the normalization factor of $\sqrt{\frac N Q} = \sqrt N$ in this case, the analogues of \cref{fig:murmN-1a,fig:murmN-1b} are graphs which individually have similar shapes, but whose vertical scale shrinks with $X$.  That the unweighted averages for $Q=1$ tend to $0$ reflects the root number symmetry $M^+_{\mathcal F}(x; \beta) = -M^-_{\mathcal F}(x; \beta)$.

\begin{figure}
\begin{minipage}{.5\textwidth}
\captionof{figure}{Weight 2 murmurations for $\sqrt N W_1$ for squarefree levels $2000 \le N \le 4000$}
 \includegraphics[width=\textwidth]{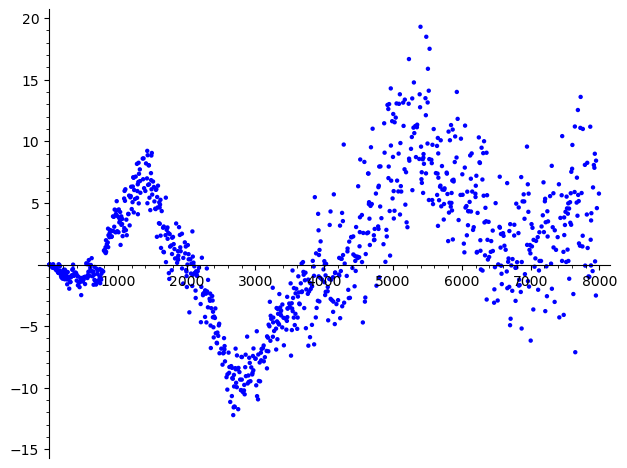}
\label{fig:murmN-1a}
\end{minipage}%
\begin{minipage}{.5\textwidth}
\captionof{figure}{Weight 2 murmurations for $\sqrt N W_1$ for squarefree levels $4000 \le N \le 8000$}
 \includegraphics[width=\textwidth]{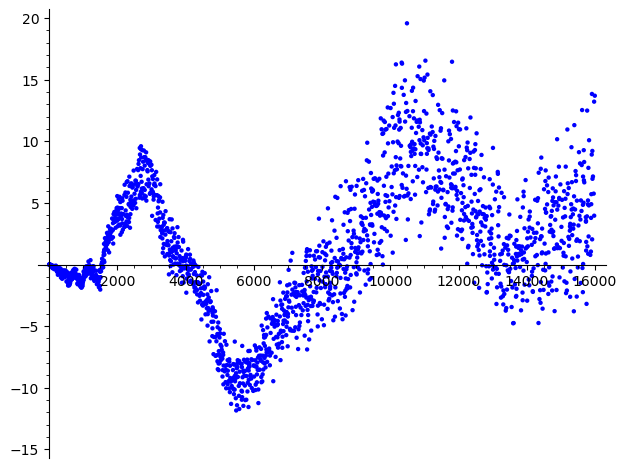}
\label{fig:murmN-1b}
\end{minipage}
\end{figure}

\begin{rem} One could also consider weighting the sums in \eqref{eq:gen-avgs} by signs, e.g., $(-1)^{\omega(N/Q)}$ in the squarefree case when $k \ge 4$, to account for the initial bias from \cref{thm11,thm16}.  However in the situations we tested including this sign actually destroys the murmurations! 
\end{rem}

In \cref{prop:Tl-large}, we write down a formula for $\tr_{S_k^\new(N)} W_Q T_\ell$ (for simplicity for squarefree $N$) which is amenable to computing such averages.  We used this to investigate murmurations with respect to Atkin--Lehner signs in a variety of settings.  What seems important for the existence of such murmurations is that one considers a sequence of $(N,Q)$'s  which are ``arithmetically compatible''.  For instance, taking a sequence of $(N, Q)$'s where $Q \approx \sqrt N$, we numerically saw a random distribution of averages with no apparent murmurations.

Let $\N^{\sqf}$ denote the set of squarefree positive integers, and for $r \ge 1$ let $\N^{\sqf}_r$ be the subset of $\N^{\sqf}$ consisting of those with exactly $r$ prime factors.
We say a sequence of pairs $\{ (N,Q) \}$ is \emph{arithmetically compatible} in any of the following situations:

\begin{enumerate}[(I)]
\item $N = Q M$ where $M$ is constant and $Q$ ranges over all elements of $\N^{\sqf}$ or $\N^{\sqf}_r$ such that $(M,Q) = 1$.

\item $N = Q M$ where $Q$ is a squarefree constant and $M$ ranges over all elements of $\N$, $\N^{\sqf}$ or $\N^{\sqf}_r$ such that $(M,Q) = 1$.

\item Fix $r \ge 2$, let $0 \le m < r$, and fix primes $p_1 < \dots < p_m$.  Let $N = p_1 \dots p_r$ range over elements of $\N^{\sqf}_r$ such that $p_1 < \dots < p_r$, and $Q = p_{i_1} \dots p_{i_s}$ where $0 \le s < r$ and $\{ i_1, \dots, i_s \}$ is a fixed subset of $\{ 1, \dots, r \}$.
\end{enumerate}
In all cases it is assumed the sequence $\{ (N,Q) \}$ is arranged in order of increasing $N$.

For instance, when $r=2$, Type III consists of sequences $\{ (N,Q) = (p_1 p_2, Q) : p_1 < p_2 \}$, where we can choose to fix $p_1$ or not, and $Q$ is taken to be one of the following 4 fixed forms: $1, p_1, p_2, p_1 p_2$.  Note that Type III includes the $\N^{\sqf}_r$ cases of Types I and II.  

In \cref{fig:murmN-1a,fig:murmN-1b} where $Q=1$, and more generally for Type II and II graphs, it is not clear whether the averages $A^{\mathcal Q}$ should actually converge to a continuous function with fluctuations in very short intervals or whether there is some inherent ``random noise.''  To be more confident the limiting graphs should exist, we consider the $\delta$-smoothed averages
\[ \tilde A^{\mathcal Q,\delta}_{\mathcal F}(\ell,X; \beta) = \frac 1{\# \{ \ell' : \ell \le \ell' < \ell + \ell^\delta \} }\sum_{\ell \le \ell' < \ell +  \ell^\delta} A^{\mathcal Q}_{\mathcal F}(\ell',Y; \beta). \]  
See \cref{fig:murmN-sm1,fig:murmN-sm2} for $\delta$-smoothed versions of \cref{fig:murmN-1b}.

\begin{figure}
\begin{minipage}{.5\textwidth}
\captionof{figure}{$\delta$-smoothed version of \cref{fig:murmN-1b} with $\delta = \frac 12$}
 \includegraphics[width=\textwidth]{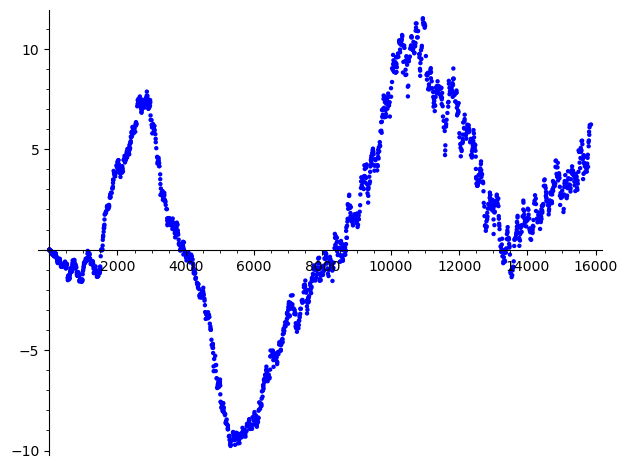}
\label{fig:murmN-sm1}
\end{minipage}%
\begin{minipage}{.5\textwidth}
\captionof{figure}{$\delta$-smoothed version of \cref{fig:murmN-1b} with $\delta = \frac 34$}
 \includegraphics[width=\textwidth]{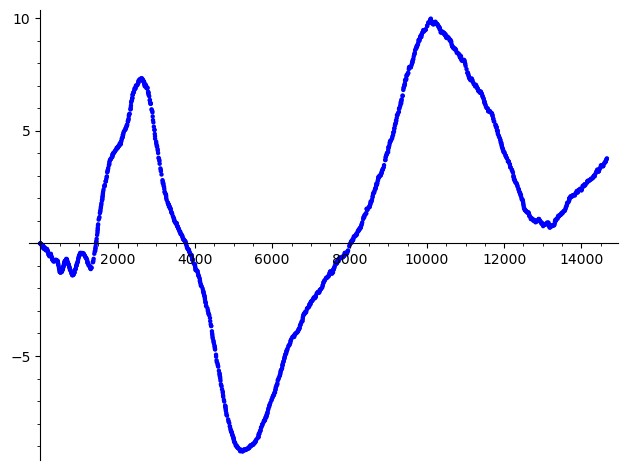}
\label{fig:murmN-sm2}
\end{minipage}%
\end{figure}

\begin{conj} [Murmurations for Atkin--Lehner operators]
\label{conj:WQ-murm}
Let $(\mathcal N, \mathcal Q)$ be an arithmetically compatible sequence of $(N,Q)$'s of Type I, II or III as above, and fix a weight $k$. 
Let $\mathcal F$ the family of newforms which lie in $S^\new_k(N)$ for some $N \in \mathcal N$. 

\begin{enumerate}
\item If $(\mathcal N, \mathcal Q)$ is of Type I, then the averages $A^{\mathcal Q}_{\mathcal F}(\ell,X)$ have murmurations which are scale invariant in $\frac \ell N$.  More precisely, as $\ell, X \to \infty$ such that $\frac \ell X \to x$, 
\[  A^{\mathcal Q}_{\mathcal F}(\ell,X; \beta) \to M_{\mathcal F}^{\mathcal Q}(x; \beta) \]
for a murmuration function $M_{\mathcal F}^{\mathcal Q}$ which is continuous on $[0,\infty) \times (1, \infty)$.

\item If $(\mathcal N, \mathcal Q)$ is of Type II or III, then for some $\delta < 1$ the $\delta$-smoothed averages $\tilde A^{\mathcal Q, \delta}_{\mathcal F}(\ell,X; \beta)$ have murmurations which are scale invariant in $\frac \ell N$.  More precisely, as $\ell, X \to \infty$ such that $\frac \ell X \to x$, 
\[  \tilde A^{\mathcal Q, \delta}_{\mathcal F}(\ell,X; \beta) \to \tilde M_{\mathcal F}^{\mathcal Q,\delta}(x; \beta) \]
for a murmuration function $\tilde M_{\mathcal F}^{\mathcal Q, \delta}$ which is continuous on $[0,\infty) \times (1, \infty)$.
\end{enumerate}
\end{conj}

Recall that Type I murmurations with $N=Q$ squarefree and $k=2$ are illustrated in \cref{fig:fixedroot1,fig:fixedroot2}.  See \cref{fig:murm5Q} for a Type I plot with $N=5Q$ ($Q$ squarefree) and $k=4$, averaging over the range $15000 < N < 30000$.
Similarly, non-smoothed and smoothed Type II plots with $Q=1$ are given in 
\cref{fig:murmN-1a,fig:murmN-1b,fig:murmN-sm1,fig:murmN-sm2}.
Illustrations of Type III situations with $N=pq$ are presented in a different format in  \cref{fig:pq-prof,fig:pq-prof4,fig:pq-all-prof4} by looking at Atkin--Lehner eigenspaces, as will be described below.

\begin{figure}
\begin{minipage}{.5\textwidth}
\captionof{figure}{Weight 4 murmurations for $N = 5Q$ squarefree}
 \includegraphics[width=\textwidth]{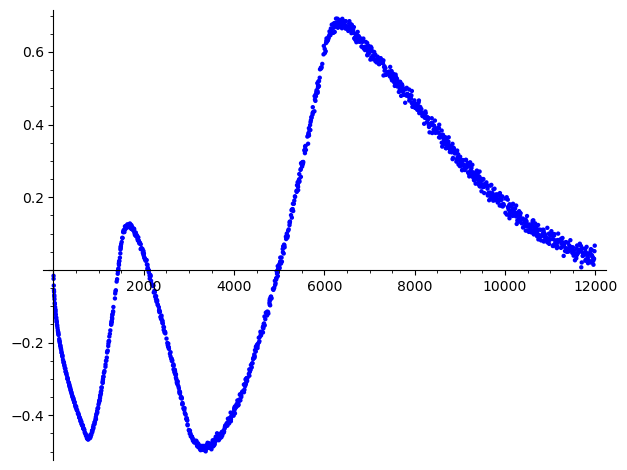}
\label{fig:murm5Q}
\end{minipage}%
\begin{minipage}{.5\textwidth}
\captionof{figure}{Murmurations for AL-eigenspaces on $S_2(2q)$}
 \includegraphics[width=\textwidth]{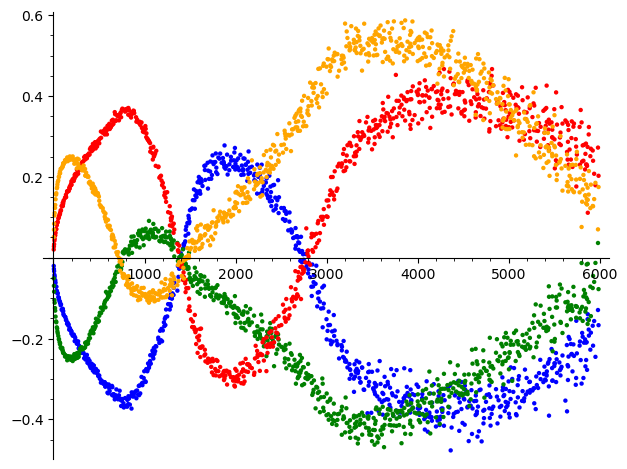}
\label{fig:pq-prof}
\end{minipage}%
\end{figure}

When $M=1$ and $Q \in \N^\sqf$, \cref{conj:WQ-murm} follows from \cite{zubrilina}, and we expect that one can prove the general Type I case in a similar way.  However, the types of sums that one needs to handle for Types II and III will require a different type of analysis.  (For Type I and given $x$, there are only finitely many terms to consider from the trace formula, but for Types II and III there are an unbounded number of terms.)  Here we merely show the following as evidence towards the above conjecture.

\begin{thm} \label{thm:WQ-murm} Assume $(\mathcal N, \mathcal Q)$ is an arithmetically compatible family of pairs of Type I, where $\mathcal N = \{ MQ : Q \in \mathbb N^\sqf, (Q, M) = 1 \}$ for some fixed squarefree $M$.  Then \cref{conj:WQ-murm}(1) holds for $x < \frac 1{4M} - \eps$, for any $\eps > 0$.  Specifically $A_{\mathcal F}^{\mathcal Q}(\ell, X; \beta) \to c \sqrt x + \delta_{k=2} d$ in this range, for some constants $c = c_{\mathcal F, \mathcal Q, \beta}$ and $d = d_{\mathcal F, \mathcal Q, \beta}$. 
\end{thm}

In particular, when $M=1$, this says that the first $\frac 1{16}$-th of the graphs in \cref{fig:murmN-1a,fig:murmN-1b} are approximately of the form $c \sqrt x + d$.  This agrees with \cite{zubrilina}, which also computes $c, d$.
Similarly, the theorem asserts that the first $\frac 1{25}$-th of the graph in \cref{fig:murm5Q} is approximately of the form $c \sqrt x$.

If one works with (squarefree or general) levels $N$ that have a fixed number of prime divisors as in Type III, then one can alternatively look how the collection of Atkin--Lehner signs is correlated with Fourier coefficients.  See \cref{fig:pq-prof,fig:pq-prof4} for a graph of averages of $a_\ell$'s over newforms in $S_2(pq)$ and $S_4(pq)$ with fixed Atkin--Lehner signs at $p, q$, where $p = 2$ and $3000 < q < 6000$.  The blue and green dots correspond to signs \verb|++| and \verb|--| and red and orange dots to signs \verb|+-| and \verb|-+|, respectively.  (The first sign denotes the sign at $p$, and the second the sign at $q$.)  See \cref{fig:pq-all-prof4} for the analogous graph for $S_4(pq)$ where $p, q$ both vary such that $p < q$ and $6000 < pq < 12000$.

\begin{conj} [Murmurations on Atkin--Lehner eigenspaces]
\label{conj:eps-murm}
Fix $k, r$ and $0 \le m < r$.  Fix primes $p_1 < \dots < p_m$.  Let $\mathcal N$ be the set of levels $N = p_1 \dots p_r \in \N^{\sqf}_r$ such that $p_1 < p_2 < \dots < p_r$.
Let $\mathcal F$ be the family of weight $k$ newforms with level in $\mathcal N$.  For $\eps = (\eps_1, \dots, \eps_r) \in {\pm 1}^r$, let $\mathcal F^\eps$ be the subset of $f \in \mathcal F$ with Atkin--Lehner sign $\eps_i$ at $p_i$, and consider the averages
\begin{equation} \label{eq:al-avgs}
A^{\eps}_{\mathcal F}(\ell,X; \beta) = \frac 1{\#\mathcal F^\eps(X,\beta X)}  \sideset{}{'}\sum_{f \in \mathcal F^\eps(X,\beta X)} \ell^{1 - \frac k2} a_\ell(f).
\end{equation}

These averages have murmurations which are scale invariant in $\frac \ell N$.  That is, as $\ell, X \to \infty$ such that $\frac \ell X \to x$, 
\[  A^\eps_{\mathcal F}(\ell,X; \beta) \to M_{\mathcal F}^{\eps}(x; \beta) \]
for a murmuration function $M_{\mathcal F}^{\eps}$ which is continuous on $[0,\infty) \times (1, \infty)$.
\end{conj}

\begin{figure}

\begin{minipage}{.5\textwidth}
\captionof{figure}{Murmurations for AL-eigenspaces on $S_4(2q)$}
 \includegraphics[width=\textwidth]{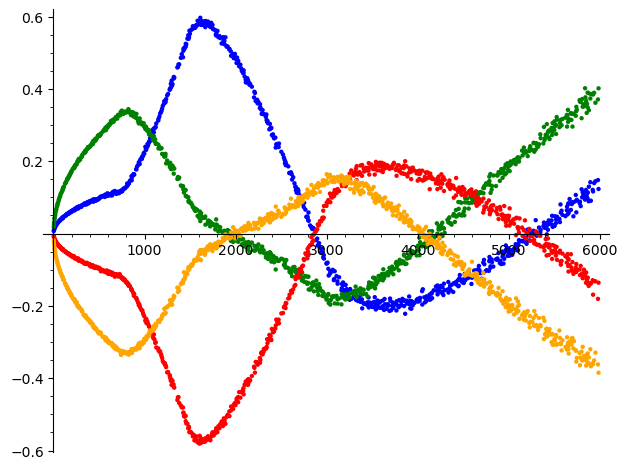}
\label{fig:pq-prof4}
\end{minipage}%
\begin{minipage}{.5\textwidth}
\captionof{figure}{Murmurations for AL-eigenspaces on $S_4(pq)$}
 \includegraphics[width=\textwidth]{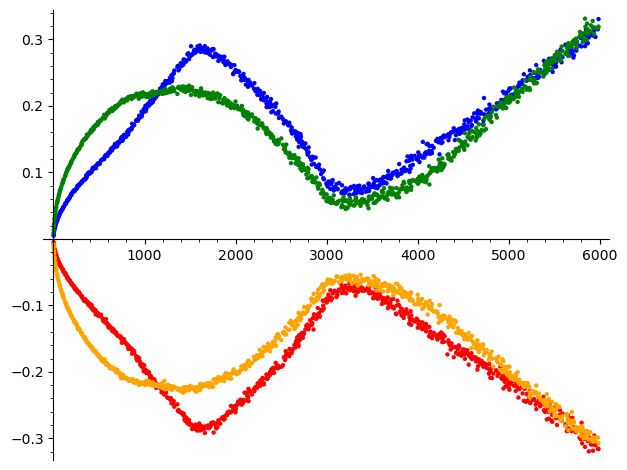}
\label{fig:pq-all-prof4}
\end{minipage}%
\end{figure}

\begin{rem} \label{rem:murm}
\begin{enumerate}
\item As in the case of murmurations with respect to global root numbers, we expect these murmuration functions oscillate infinitely, and at least for Type I that they arise from murmuration density functions which correspond to letting $\beta \to 0$ as $X \to \infty$.

\item In the Type I case of \cref{conj:WQ-murm}, there is no need to weight by $\sqrt{N/Q}$ as $N/Q$ is constant.  For Type II, one is  weighting by a constant times $\sqrt N$, so by scale-invariance one could alternatively weight by $\sqrt \ell$.  We chose to weight by $\sqrt{N/Q}$ as it seems to be the right order of normalization in general. 

\item When averaging Fourier coefficients over Atkin--Lehner eigenspaces $\mathcal F^\eps$ in \cref{conj:eps-murm}, there is no need to weight by an analogue of $\sqrt{N/Q}$ (or consider smoothed averages), since the trace of $T_\ell$ on $\mathcal F^\eps(N)$ is a linear combination of the traces of $T_\ell W_Q$ on $\mathcal F(N)$ where one sums over all $Q \mid N$.  Correspondingly, we expect the murmuration functions to be different on each Atkin--Lehner eigenspace when $r=m+1$, i.e., when all but one prime is fixed in the level as in \cref{fig:pq-prof,fig:pq-prof4}.

\item It is not clear whether the smoothed averages are actually needed in \cref{conj:WQ-murm}(2), or how much smoothing is actually needed.

\item One could also consider analogues where $Q$ is not required to be squarefree.
\end{enumerate}
\end{rem}

See \cref{sec:eps-murm} for details on how \cref{conj:eps-murm} is related to \cref{conj:WQ-murm}.  This relation implies that \cref{thm:WQ-murm} also provides evidence for  \cref{conj:eps-murm}.

Finally, one might wonder about analogues of \cref{conj:WQ-murm,conj:eps-murm} in the original setting of elliptic curves.  Earlier calculations of Sutherland indicate that there are no apparent murmurations if one does not weight by any root number; more generally our calculations also do not suggest any murmurations for elliptic curves in Type II situations.  

However, numerically there appear to be murmurations in Type I situations, i.e., $N=QM$ with $M$ fixed and $Q$ varying, at least after smoothing.
For instance, see \cref{fig:EC2M} for a plot of $\tilde A^{\mathcal Q}_{\mathcal E}(\ell,X;2)$ for the family with $N = 2Q$ squarefree, $X = 20000$ and $\beta = 2$, and \cref{fig:EC2M-sm} for the smoothed averages $\tilde A^{\mathcal Q,\delta}_{\mathcal E}(\ell,X)$ with $\delta = 0.75$.  For comparison, these averages (in blue) are plotted on top of the averages \eqref{eq:drew-avgs} weighted by global root numbers (in red).  This suggests the following.

\begin{figure}
\begin{minipage}{.5\textwidth}
\captionof{figure}{$w_Q$ (blue) versus $w_N$ (red) murmurations for elliptic curves of squarefree conductor $N = 2Q$, with $20000 < Q < 40000$}
 \includegraphics[width=\textwidth]{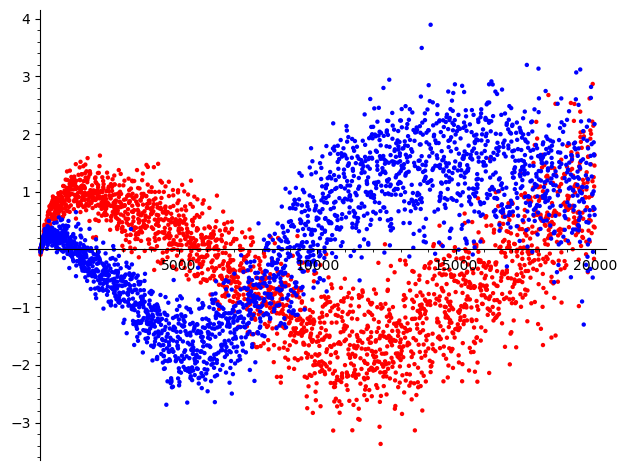}
\label{fig:EC2M}
\end{minipage}%
\begin{minipage}{.5\textwidth}
\captionof{figure}{$\delta$-smoothed analogues of \cref{fig:EC2M} with $\delta = \frac 34$}
 \includegraphics[width=\textwidth]{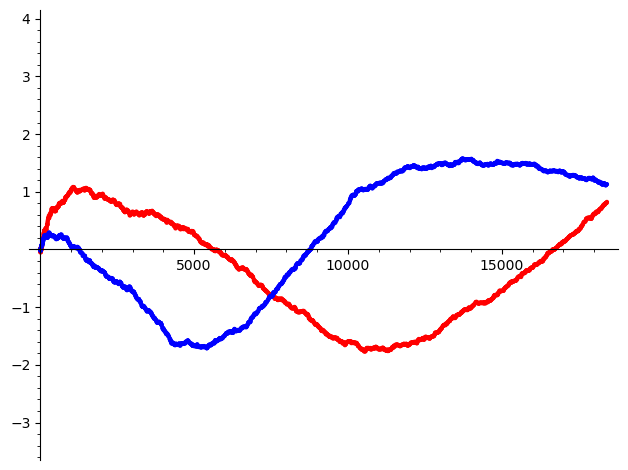}
\label{fig:EC2M-sm}
\end{minipage}%
\end{figure}

\begin{conj} [Partial root number murmurations for elliptic curves]
\label{conj:EC-murm}
Let $(\mathcal N, \mathcal Q)$ be an arithmetically compatible sequence of $(N,Q)$'s of Type I. 
Let $\mathcal E$ the set of rational newforms which lie in $S_2(N)$ for some $N \in \mathcal N$.  Then for some $\delta < 1$, the smoothed averages  $\tilde A^{\mathcal Q, \delta}_{\mathcal E}(\ell,X; \beta)$ have murmurations which are scale invariant in $\frac \ell N$.
\end{conj}

As in the Type II and III cases for modular forms, it is not clear whether the smoothing in \cref{conj:EC-murm} should be needed (even in the global root number case of $Q=N$).

\subsection{Additional remarks}
We checked the results stated in the introduction, as well as many of our formulas below, numerically in Sage \cite{sage} for a wide variety of small parameters.

\subsection*{Acknowledgements}
I thank John Bergdall, Alex Cowan, Nina Zubrilina, and the anonymous referee for some helpful comments and discussions.
Some of the computing for this project was performed at the OU Supercomputing Center for Education \& Research (OSCER) at the University of Oklahoma (OU). 


\section{Notation and preliminaries}


\subsection{Class numbers}

Let $\Delta = \lambda^2 \Delta_0$, where $\Delta_0$ is a negative fundamental discriminant.  Let $h'(\Delta)$ be the weighted class number for primitive binary negative definite quadratic forms of discriminant $\Delta$.  Explicitly, $h'(-3) = \frac 13$, $h'(-4) = \frac 12$, and $h'(\Delta_0) = h(\Delta_0)$ is the usual class number if $\Delta_0 < -4$.
Moreover, 
\begin{equation} h'(\lambda^2 \Delta_0) = \gamma_{\Delta_0}(\lambda) h'(\Delta_0),
\quad \gamma_{\Delta_0}(\lambda) = \sum_{d \mid \lambda} \mu(d) {\Delta_0 \leg d} \frac \lambda d.
\end{equation}
Since $\gamma_{\Delta_0}$ is a Dirichlet convolution of multiplicative functions, it is also multiplicative, and given on (nontrivial) prime powers by
\[ \gamma_{\Delta_0}(p^m) = 
p^{m-1} \left( p - {\Delta_0 \leg p} \right). \]

Let $\Delta \le 0$ and $t \ge 1$.  Let $H_1(\Delta) = H(\Delta)$ be the Hurwitz class number.  This is defined by
\[ H(\Delta) = \sum_{d \mid \lambda} h'(d^2 \Delta_0), \]
where we write $\Delta = \lambda^2 \Delta_0$ for a fundamental discriminant $\Delta_0$.
One deduces (e.g., \cite[(2.3)]{me:rootno}) that
\begin{equation} \label{eq:Hlambda2}
 H(\lambda^2 \Delta_0) =  \eta_{\Delta_0}(\lambda) h'(\Delta_0), \quad
\eta_{\Delta_0}(\lambda) = \sum_{d \mid \lambda} \mu(d) {\Delta_0 \leg d} \sigma(\lambda/d).
\end{equation}
Just like $\gamma_{\Delta_0}$, we have that $\eta_{\Delta_0}$ is multiplicative and it is given on prime powers by
\[ \eta_{\Delta_0}(p^m) = 
 \sigma(p^m) - {\Delta_0 \leg p} \sigma(p^{m-1}). \]

For $t \ge 2$, write $(t,\Delta) = a^2 b$ where $b$ is squarefree, and put $\Delta' = \Delta/(t,\Delta)$, $t' = t/(t,\Delta)$.  Set
\[ H_t(\Delta) =
\begin{cases} (t,\Delta) {\Delta'/b \leg t'} H(\Delta'/b) & \text{if } b \mid \Delta', \\
0 & \text{else}.
\end{cases} \]


\subsection{Other quantities arising in the trace formula}
Fix $s \in \Z$ and $\ell \ge 1$.  Let $\rho, \bar \rho$ denote the roots of $x^2 - s X + \ell$.  Define
\[ p_k(s,\ell) =
\begin{cases}
\frac{\rho^{k-1} - \bar \rho^{k-1}}{\rho - \bar \rho} & \text{if }s^2 \ne 4 \ell, \\
(k-1) ( \frac s 2 )^{k-2} &\text{if } s^2 = 4 \ell.
\end{cases} \]
In particular, when $s=0$, the roots $\rho, \bar \rho$ are $\pm \sqrt{-\ell}$, and 
\[ p_k(0,\ell) = (-\ell)^{\frac k2 - 1}. \]
We also remark that
\[ p_2(s,\ell) = 1 \]
and
\[ p_k(s, \ell) = \ell^{\frac k2 - 1} U_{k-2}(\frac s{2 \sqrt \ell}) , \] 
where $U_k(t)$ denotes the Chebyshev polynomial of the second kind.

Let $Q(n)$ be the greatest integer such that $Q(n)^2 \mid n$.


\section{Traces on newspaces} \label{sec:tnew}


Fix an even weight $k \ge 2$,  a prime $q$ and positive integers 
$\ell, M$.  Assume $q, \ell, M$ are pairwise coprime.  
The Atkin--Lehner operator $W_q$, defined as in \cite{AL},
acts on $S_k(N)$.  This action is taken to be the trivial action when $(q,N) = 1$.

For an integer $r \ge 0$, set
\[ t(r; M) = \tr_{S_k(q^r M)} T_\ell W_q, \quad
 t^\new(r; M)  =  \tr_{S_k^\new(q^r M)} T_\ell W_q. \]
 (Throughout the analysis in this section, $q, \ell, k$ will be fixed, so we suppress them from our notation for brevity.)
 For $r < 0$, we interpret these quantities to be $0$.  When $\ell = 1$,
 $t^\new(r;M) = \Delta_k(q^r, M)$.

Fix a newform $g \in S_k(q^{r_0}M_0)$.  Denote by $w_q(g)$ the eigenvalue for $g$ under the action of $W_q$ on $S_k(q^{r_0}M_0)$.
Let $N = q^r M$, and assume that $r_0 \le r$ and
$M_0 \mid M$.  Let $\pi_g^N$ denote the subspace of $S_k(N)$ spanned by forms $g(dz)$ where $d \mid q^{r-r_0} M_0^{-1} M$.  Then $T_\ell$ acts by a scalar on $\pi_g^N$.
One computes the trace of $W_q$ on $\pi_g^N$ from \cite[(5.1)--(5.2)]{AL}, which yields
\[ \tr_{\pi_g^N} T_\ell W_q = 
\begin{cases}
\sigma_0(M/M_0) a_\ell(g) w_q(g) &\text{if } r \equiv r_0 \mod 2, \\
0 & \text{else}.
\end{cases} \]

Thus
\[ t(r; M) = \sum_{\substack{r_0 \le r \\
r_0 \equiv r \mod 2}} \sum_{M_0 \mid M} \sigma_0(M/M_0) t^\new(r_0; M_0). \]
Hence
\[  t(r; M) - t(r-2; M) =  \sum_{M_0 \mid M} \sigma_0(M/M_0) t^\new(r; M_0). \]

Since $\sigma_0 = 1 * 1$ (where $*$ denotes Dirichlet convolution) and the Dirichlet inverse of the constant function $1$ is $\mu$, the Dirichlet inverse of $\sigma_0$ is $\mu * \mu$, which is the multiplicative function defined by
\[ (\mu * \mu)(p^m) = \begin{cases}
-2 & \text{if }m = 1, \\
1 & \text{if }m = 2, \\
0 & \text{if }m \ge 3.
\end{cases} \]
Thus 
\[ t^\new(r,M) = \sum_{d \mid M} (\mu * \mu) (d) \left( t(r; M/d) - t(r-2; M/d) \right). \]

Reorganizing the trace formula from \cite[(2.7)]{SZ}, we see that
\begin{equation}
t(r; M) = A_{1,0}(r; M) - \delta_{r \ge 2} \, A_{1,1}(r-2; M) + A_2(r; M) + A_3,
\end{equation}
where
\[ A_{1,\eps}(r;M) = -\frac 12 \sum_{\substack{s^2 \le 4 q^r  \ell \\ q^{r+\eps} \mid s}} p_k(q^{-r/2} s, \ell) \sum_{\substack{t \mid M \\ M/t \text{ squarefree}}}  H_t(s^2 - 4 q^r \ell ), \]
\[ A_2(r; M) = - \frac 12 \delta_{r \text{ even}} \, \phi(q^{r/2}) \sum_{\substack{\ell' \mid \ell \\ q^{r/2} \mid (\ell' + \ell/\ell')}} \min (\ell', \ell/\ell')^{k-1} \sum_{\substack{t \mid M \\ M/t \text{ squarefree}}}
 (Q(t), (\ell' - \ell/\ell')),  \]
 and
 \[ A_3 = A_3(r; M) = \delta_{k = 2} \, \sigma(\ell). \]
 Here and below the index $s$ lies in $\Z$, whereas $t, \ell' \in \Z_{> 0}$.

Hence
\begin{equation*} t(r; M) - t(r-2; M) = 
\begin{cases}
A_{1,0}(r;M) + A_2(r; M) + A_3 & \text{if }r = 0, 1, \\
A_{1,0}(r;M) - A_{1,0}(r-2;M) - A_{1,1}(r-2;M) \\
\quad {} +  \delta_{r \ge 4} A_{1,1}(r-4;M)  + A_2(r;M) - A_2(r-2;M) & \text{if } r \ge 2. \\ 
\end{cases}
\end{equation*}

To compute $t^\new(r;M)$, we want to compute the quantities
\begin{equation} \tilde A_\star(r; M) := \sum_{d \mid M} (\mu * \mu)(d) A_\star(r;M/d) = ((\mu * \mu) * A_{\star}(r; \bigdot )) (M),
\end{equation}
where $\star$ is any index.

\begin{prop} With notation as above, we have
\begin{equation}
\label{eq:tnewcases}
 t^\new(r;M) =
\begin{cases}
\tilde A_{1,0}(r;M) +\tilde  A_2(r; M) + \tilde A_3 & \text{if } r = 0, 1, \\
\tilde A_{1,0}(r;M) - \tilde A_{1,0}(r-2;M) - \tilde A_{1,1}(r-2;M) \\
\quad {} + \delta_{r \ge 4} \tilde A_{1,1}(r-4;M) + \tilde A_2(r;M) - \tilde A_2(r-2;M) & \text{if } r \ge 2. \\ 
\end{cases}
\end{equation}
\end{prop}

In the remainder of this section, we compute the quantities $\tilde A_\star(r; M)$ in the main situations of interest for us.

\subsection{$A_{1,\eps}$ sums} Let $\eps \in \{ 0, 1 \}$.

\subsubsection{Computing $\tilde A_{1,\eps}(r;M)$ when $4 \ell < q^{r+2\eps}$} \label{sec:A1a}
Suppose $4 \ell < q^{r+2\eps}$.  Then only the $s=0$ term occurs in the outer sum for $A_{1,\eps}(r;M)$, and we have
\[ A_{1,\eps}(r; M) = - \frac 12 (-\ell)^{\frac k2 - 1} \sum_{\substack{t \mid M \\ M/t \text{ squarefree}}} H_t(-4 q^r \ell). \]
We assume $(M, q\ell) = 1$, so for $t \mid M$ we have $(t, -4 q^r \ell) = 1$ unless $t$ is even, in which case 
$(t, -4 q^r \ell) = 2$ or $4$.  Hence
\begin{equation}
\label{eq:Htcases}
 H_t(-4 q^r \ell) = 
\begin{cases}
{ - q^r \ell \leg t} H(-4 q^r \ell) &\text{if } t \text{ odd}, \\
2{ - q^r \ell \leg t/2} H(-q^r \ell) & \text{if }t \equiv 2 \mod 4, \\
4{ - q^r \ell \leg t/4} H(-q^r \ell) & \text{if }t \equiv 0 \mod 4.
\end{cases} 
\end{equation}

In the following analysis, we will assume $\ell$ is squarefree.  Write $-4q^r \ell = \lambda^2 \Delta_0$, where $\Delta_0$ is a fundamental discriminant.  We can rewrite $H(-4q^r \ell)$ and $H(-q^r \ell)$ in terms of $h'(\Delta_0)$ on a case-by-case basis as in \cref{tab1}.
We will use this to calculate $\tilde A_1$ in cases.

The following computation will be useful.  First, for an integer $\Delta$, note that
${\Delta \leg \bigdot} * |\mu|$ is the multiplicative function given on nontrivial prime powers by
\[ ({\Delta \leg \bigdot} * |\mu|) (p^m) =
\begin{cases}
1 + {\Delta \leg p} & \text{if }p \nmid \Delta \text{ or } m = 1, \\
0 & \text{if } p \mid \Delta \text{ and } m=1.
\end{cases} \]
Consequently $\kappa_\Delta := (\mu * \mu) * ({\Delta \leg \bigdot} * |\mu|)$ is the multiplicative function given on nontrivial prime powers by
\begin{equation} \label{eq:kappa-def}
 \kappa_\Delta(p^m) =
\begin{cases}
{\Delta \leg p} - 1 & \text{if } m = 1, \\
-{\Delta \leg p}  & \text{if }p \nmid \Delta \text{ and } m = 2,  \\
-1 & \text{if }p \mid \Delta \text{ and } m = 2,  \\
1 & \text{if } p \mid \Delta \text{ and } m = 3,  \\
0 & \text{else}.
\end{cases}
\end{equation}

\begin{table}
\renewcommand{\arraystretch}{1.333}
\begin{tabular}{l | C | C | C | C }
& \Delta_0 & \lambda & H(-4q^r \ell) & H(-q^r \ell) \\
\hline
$r$ even, $\ell \not \equiv 3 \mod 4$ & -4 \ell & q^{r/2} & \eta_{-4\ell}( q^{r/2}) h'(-4\ell) & 0 \\
$r$ even, $\ell  \equiv 3 \mod 4$ & - \ell & 2q^{r/2} & \eta_{-\ell}(2q^{r/2}) h'(-\ell) & \eta_{-\ell}(q^{r/2}) h'(-\ell) \\
$r$ odd, $q \ell \not \equiv 3 \mod 4$ & -4 q\ell & q^{(r-1)/2} & \eta_{-4q\ell}( q^{(r-1)/2}) h'(-4q\ell)& 0 \\
$r$ odd, $q \ell  \equiv 3 \mod 4$ & -q\ell & 2q^{(r-1)/2} &  \eta_{-q\ell}(2q^{(r-1)/2}) h'(-q \ell)& \eta_{-q\ell}(q^{(r-1)/2}) h'(-q \ell) \\
\end{tabular}
\caption{Hurwitz class numbers by case}
\label{tab1}
\renewcommand{\arraystretch}{1}
\end{table}

When $M$ is odd, this is simple: $H_t(-4 q^r \ell)$ will be given by the first case of
\eqref{eq:Htcases}.  From this we see
\[ A_{1,\eps}(r; \bigdot) = - \frac 12 (-\ell)^{\frac k2 - 1} \, H(-4 q^r \ell) \cdot {-q^r \ell \leg \bigdot} * |\mu| \qquad (M \text{ odd}) \]
Therefore 
\[ \tilde A_{1,\eps}(r; M) = - \frac 12 (-\ell)^{\frac k2 - 1} \, H(-4 q^r \ell) \, \kappa_{-q^r \ell}
\qquad (M \text{ odd}). \]

For general $M$, we will evaluate $\tilde A_{1,\eps}$ by separating it into 3 sums as follows:
\begin{multline*}
 \tilde A_{1,\eps}(r; M) = \sum_{\substack{d \mid M \\ d \text{ odd}}} (\mu * \mu)(d) A_{1,\eps}(r;M/d) \\
+ \sum_{\substack{d \mid M \\ d \equiv 2 \mod 4}} (\mu * \mu)(d) A_{1,\eps}(r;M/d) 
+ \sum_{\substack{d \mid M \\ d \equiv 4 \mod 8}} (\mu * \mu)(d) A_{1,\eps}(r;M/d) 
\end{multline*}
We will also separate the sum in $A_{1,\eps}$ over $t$ according to $v_2(t)$.  Write $M = 2^e M'$ where $M'$ is odd.  Then
\[ A_{1,\eps}(r; M) = - \frac 12 (-\ell)^{\frac k2 - 1} \sum_{\substack{t \mid M' \\ M'/t \text{ squarefree}}} \left( H_{2^{e}t}(-4 q^r \ell)  + H_{2^{e-1}t}(-4 q^r \ell) \right), \]
where we interpret $H_t$ to be $0$ if $t \not \in \Z$.
Hence
\[  \tilde A_{1,\eps}(r; M) = - \frac 12 (-\ell)^{\frac k2 - 1} 
\sum_{j=0}^2 \sum_{d \mid M'} (\mu * \mu)(2^j d) 
\sum_{\substack{t \mid (M'/d) \\ M'/dt \text{ squarefree}}} \left( H_{2^{e-j}t}(-4 q^r \ell)  + H_{2^{e-j-1}t}(-4 q^r \ell) \right). \]

Rewriting \eqref{eq:Htcases} with index $2^e t$ for $t$ odd, we have
\begin{equation*}
 H_{2^e t}(-4 q^r \ell) = 
\begin{cases}
{ - q^r \ell \leg t} H(-4 q^r \ell) & \text{if }e = 0, \\
2{ - q^r \ell \leg t} H(-q^r \ell) & \text{if }e = 1, \\
4{ - q^r \ell \leg t} { - q^r \ell \leg 2}^{e-2} H(-q^r \ell) & \text{if }e \ge 2.
\end{cases} 
\end{equation*}
Applying this to the previous formula yields
\begin{equation} \label{eq:AA1-lsmall}
\tilde A_{1,\eps}(r; M) = - \frac 12 (-\ell)^{\frac k2 - 1} \alpha_1(-q^r \ell; e) 
\kappa_{-q^r \ell}(M'), \quad M = 2^e M' \text{ with } M' \text{ odd},
\end{equation}
where
\[ \alpha_1(-q^r\ell; e) =
\begin{cases}
H(-4q^r \ell) & \text{if }e = 0, \\
2H(-q^r \ell) - H(-4 q^r \ell) & \text{if }e = 1, 2, \\
\left( 4 {-q^r \ell \leg 2} - 6 \right) H(-q^r \ell) + H(-4 q^r \ell) & \text{if }e = 3, \\
\left( 2 - 4 {-q^r \ell \leg 2} \right) H(-q^r \ell) & \text{if }e = 4, \\
0 & \text{if }e \ge 5.
\end{cases} \]

\subsubsection{Computing $\tilde A_{1,0}(r;M)$ when $\ell = 1$}
Assume $\ell = 1$.  Then \cref{sec:A1a} computes $\tilde A_{1,0}(r; M)$ in all cases except when $q^r \le 4$, where there are terms with $s \ne 0$ in $A_{1,0}(r; \bigdot)$.  Namely, when $q^r = 1$ there are terms for $s = 0, \pm 1, \pm 2$, and 
when $q^r \in \{ 2, 3, 4 \}$ there are terms for $s = 0, \pm q^r$. 

Note that
\[ p_k(\pm 1, 1) = \frac{\zeta_6^{k-1} + \zeta_3^{k-1}}{\zeta_6+\zeta_3} =
\begin{cases}
-1 & \text{if }k \equiv 0 \mod 6, \\
1 & \text{if }k \equiv 2 \mod 6, \\
0 & \text{if }k \equiv 4 \mod 6;
\end{cases} \]
\[ p_k(\pm 2, 1) = k-1; \]
\[ p_k(\pm \sqrt 2, 1) =  
\begin{cases}
-1 & \text{if }k \equiv 0, 6 \mod 8, \\
1 & \text{if }k \equiv 2, 4 \mod 8;
\end{cases} \]
and
\[ p_k(\pm \sqrt 3, 1) =  \begin{cases}
-1 & \text{if }k \equiv 0, 8 \mod 12, \\
1 & \text{if }k \equiv 2, 6 \mod 12, \\
2 & \text{if }k \equiv 4 \mod 12, \\
-2 & \text{if }k \equiv 10 \mod 12.
\end{cases} \]



\medskip
\noindent
\underline{Case 1: $q^r = 1$}

Suppose $r=0$.  Note that
\begin{align*} A_{1,0}(0; M) &= \frac 12 \sum_{\substack{t \mid M \\ M/t \text{ squarefree}}} 
\left( (-1)^{\frac k2} H_t(-4) - 2 p_k(1,1) H_t(-3) - 2(k-1) H_t(0) \right) \\
&= \frac 1{12} \sum_{\substack{t \mid M \\ M/t \text{ squarefree}}} 
\left( 3(-1)^{\frac k2} {-4 \leg t}  - 4 p_k(1,1) {-3 \leg t} + (k-1)t \right).
\end{align*}

Hence
\begin{equation} \tilde A_{1,0}(0; M) =  \frac 1{12} \left( 3(-1)^{\frac k2} \kappa_{-4}(M) -
4 p_k(1,1) \kappa_{-3}(M) + (k-1) \kappa_\infty(M) \right), 
\end{equation}
where $\kappa_\infty = (\mu * \mu) * (\textrm{id} * |\mu|)$, which is the multiplicative function given by
\begin{equation} \label{eq:kapinfty}
 \kappa_\infty(p^m) = 
\begin{cases}
p-1 & \text{if }m = 1, \\
p^2 - p - 1 & \text{if }m = 2, \\
p^{m-3} (p-1)^2 (p+1) & \text{if }m \ge 3. 
\end{cases} 
\end{equation}

\medskip
\noindent
\underline{Case 2: $q^r = 2$}

Suppose $q=2$ and $r=1$.  Then
\[ A_{1,0}(1; M) = \sum_{\substack{t \mid M \\ M/t \text{ squarefree}}} 
\left( \frac{(-1)^{\frac k2}}2 H_t(-8) -  p_k(\sqrt 2,1) H_t(-4) \right). \]
By assumption $t \mid M$ implies $t$ is odd, so
\begin{equation} \tilde A_{1,0}(1; M) =  \frac 1{2} \left( (-1)^{\frac k2} \kappa_{-2}(M) -  p_k(\sqrt 2,1) \kappa_{-1}(M) \right). 
\end{equation}

\medskip
\noindent
\underline{Case 3: $q^r = 3$}

Suppose $q=3$ and $r=1$.  Then
\[ A_{1,0}(1; M) = \sum_{\substack{t \mid M \\ M/t \text{ squarefree}}} 
\left( \frac{(-1)^{\frac k2}}2 H_t(-12) -  p_k(\sqrt 3,1) H_t(-3) \right). \]
Hence by the same argument as above, we have
\begin{equation}
\tilde A_{1,0}(1; M) =  \left( \frac{(-1)^{\frac k2}}2 \alpha_1(-3;e) \kappa_{-3}(M') -  \frac 13 \, p_k(\sqrt 3,1) \kappa_{-3}(M) \right),
\end{equation}
where $M = 2^e M'$ with $M'$ odd.  We note that
$\alpha_1(-3; e) = \frac 43, -\frac 23, -2, 2, 0$ for $e = 0$, $e= 1,2$,
$e = 3$, $e=4$ and $e\ge 5$, respectively.

\medskip
\noindent
\underline{Case 4: $q^r = 4$}

Suppose $q=2$ and $r=2$.  Then
\begin{align*} A_{1,0}(2; M) &= \sum_{\substack{t \mid M \\ M/t \text{ squarefree}}} 
\left( \frac{(-1)^{\frac k2}}2 H_t(-16) -  (k-1) H_t(0) \right) \\
&= \frac 1{12} \sum_{\substack{t \mid M \\ M/t \text{ squarefree}}} 
\left( 9 {(-1)^{\frac k2}} {-1 \leg t} +  (k-1) t \right),
\end{align*}
whence
\begin{equation}
\tilde A_{1,0}(2; M) = \frac 1{12}  
\left( 9 {(-1)^{\frac k2}} \kappa_{-1}(M) +  (k-1) \kappa_\infty(M) \right).
\end{equation}

\subsubsection{Computing $\tilde A_{1,1}(r;M)$ when $\ell = 1$}
Assume $\ell = 1$.  Then \cref{sec:A1a} computes $\tilde A_{1,1}(r;M)$ except in the case that $q=2$ and $r=0$, so suppose this.  One sees
\begin{align*} A_{1,1}(0; M) &= \sum_{\substack{t \mid M \\ M/t \text{ squarefree}}} 
\left( \frac{(-1)^{\frac k2}}2 H_t(-4) -  (k-1) H_t(0) \right) \\
&= \frac 1{12} \sum_{\substack{t \mid M \\ M/t \text{ squarefree}}} 
\left( 3 {(-1)^{\frac k2}} {-1 \leg t} +  (k-1) t \right).
\end{align*}
This implies
\begin{equation}
\tilde A_{1,1}(0; M) = \frac 1{12}  
\left( 3 {(-1)^{\frac k2}} \kappa_{-1}(M) +  (k-1) \kappa_\infty(M) \right).
\end{equation}

\subsection{$A_2$ sums when $\ell = 1$}

Suppose $\ell = 1$.  Then
\[ A_2(r; M) = - \frac 12 \delta_{r \text{ even}} \, \delta_{q^r \mid 4} \,  \sum_{\substack{t \mid M \\ M/t \text{ squarefree}}} Q(t).  \]
The sum on the right is in fact $(Q * |\mu|)(M)$, which is a multiplicative function of $M$.  Consequently,
\[ \alpha_2 = (\mu * \mu) * (Q * | \mu |) \]
is a multiplicative function, and one can check that it is given on (nontrivial) prime powers by
\[ \alpha_2(p^m) = 
\begin{cases}
0 & \text{if }m \text{ is odd}, \\
p-2 & \text{if }m = 2,  \\
p^{\frac{m-4}2}(p-1)^2 & \text{if }m \ge 4 \text{ is even}.
\end{cases} \]

Hence, when $\ell = 1$,
\begin{equation}
\tilde A_2(r;M) =
\begin{cases}
- \frac 12 \alpha_2(M) & \text{if } q^r \in \{ 1, 4 \} \text{ and } M \in \square, \\
0 & \text{else}.
\end{cases}
\end{equation}

\subsection{$A_3$ sums}
Since $A_3(r;M) = \delta_{k=2} \, \sigma(\ell)$ is independent of $r$ and $M$,
\[ (\mu * \mu) * A_3(r; \bigdot) = \delta_{k=2} \, \sigma(\ell) ( \mu * \mu * 1)
= \delta_{k=2} \, \sigma(\ell) \cdot \mu. \]
In other words, 
\begin{equation} \label{eq:AA3}
\tilde A_3 = \tilde A_3(r;M) = \delta_{k=2} \, \mu(M) \sigma(\ell).
\end{equation}


\section{Dimension formulas} \label{sec:dim}


Let $q, r, M, k$ be as in the previous section.
Here we put together our calculations of $\tilde A_\star(r;M)$ when $\ell = 1$ to compute
\[  \Delta_k(q^r, M) = t^\new(r;M) = \dim S_k^\new(q^r M)^{+_q} - \dim S_k^\new(q^r M)^{-_q} \]
Since one knows a formula for $\dim S_k^\new(q^r M) =  \dim S_k^\new(q^r M)^{+_q} + \dim S_k^\new(q^r M)^{-_q}$ (\cite{martin}), this will imply a formula for
\[ \dim S_k^\new(q^r M)^{\pm_q} = \pm \frac 12 \left( \dim S_k^\new(q^r M) \pm \Delta_k(q^r,M) \right). \]

We will use the following explicit calculations of $\kappa_\Delta$ and $\alpha_1$.

Since $(-q^r, M) = 1$, we have
\[ \kappa_{-q^r}(M) = \begin{cases}
\prod_{p^2 \parallel M} \left(- {-q^r \leg p} \right) \prod_{p \parallel M}
\left( {-q^r \leg p} - 1 \right) & \text{if }M \text{ is cubefree}, \\
0 & \text{if }p^3 \mid M \text{ for some } p.
\end{cases}
\]
In particular $\kappa_{-q^r}(M) = 0$ if and only if
(i) ${-q^r \leg p} = 1$ for some $p \parallel M$; or
(ii) $v_p(M) \ge 3$ for some $p$.
Assuming neither (i) nor (ii) hold, then $\kappa_{-q^r}(M) = (-1)^{\omega_2(-q^r; M)} (-2)^{\omega_1(M)}$,
where $\omega_1(M)$ is the number of primes sharply dividing $M$
and $\omega_2(n; M)$ is the number of $p^2 \parallel M$ such that
${n \leg p} = 1$.

We also tabulate the values of $\alpha_1(-q^r; e)$ by cases in \cref{tab:alpha1}.
These calculations use the fact that $H(-4q^r) = (3-{-q \leg 2}) H(-q^r)$ when $-q^r \equiv 1 \mod 4$.  In particular, one sees that for $q$ odd, $\alpha_1(-q^r; e) = 0$ if and only if
(i)  $q^r \equiv 1 \mod 4$ and $e \ge 4$; (ii)  $q \equiv 3 \mod 8$,  $r$ is odd and $e \ge 5$; or (iii) $q \equiv 7 \mod 8$,  $r$ is odd and $e > 0$.  We note that $\alpha_1(-q^r; e) \le 0$ if (a) $e \ne 1, 2$ or (b) $e=3$, $r$ is odd, and $q \equiv 7 \mod 8$.  Otherwise $\alpha_1(-q^r; e) \ge 0$.

\begin{table}
\renewcommand{\arraystretch}{1.333}
\begin{tabular}{C | C | C | C | C }
e & q^r \equiv 1 \mod 4 & q^r \equiv 3 \mod 4 & q=2  \\
\hline
0 & H(-4q^r) & (3-{-q \leg 2}) H(-q^r) & H(-2^{r+2}) \\
1, 2 &  -H(-4q^r) & \left({-q \leg 2} - 1 \right) H(-q^r) & \text{---} \\
3 & H(-4q^r) &  3\left({-q \leg 2} - 1 \right) H(-q^r)  & \text{---}\\
4 &  0 & 2\left(1 - 2{-q \leg 2} \right) H(-q^r) & \text{---} \\
\ge 5 & 0 & 0 & \text{---}
\end{tabular}
\caption{Computing $\alpha_1(-q^r;e)$ by cases}
\label{tab:alpha1}
\renewcommand{\arraystretch}{1}
\end{table}

\subsection{Dimensions for $r = 1$}
First suppose $r=1$.  Then $\Delta_k(q,M) = \tilde A_{1,0}(r;M) + \tilde A_3$ by \eqref{eq:tnewcases}.

When $q \ge 5$,
\[ \Delta_k(q,M) =
 \frac 12 (-1)^{\frac k2} \alpha_1(-q; e) 
\kappa_{-q}(M') + \delta_{k=2}\mu(M), \quad M = 2^e M' \text{ with } M' \text{ odd}. \]

Otherwise
\[ \Delta_k(q,M) = 
\begin{cases}
\frac 1{2} \left( (-1)^{\frac k2} \kappa_{-2}(M) -  p_k(\sqrt 2,1) \kappa_{-1}(M) \right) + \delta_{k=2}\mu(M) & \text{if } q = 2, \\
 \left( \frac{(-1)^{\frac k2}}2 \alpha_1(-3;e) \kappa_{-3}(M') -  \frac 13 \, p_k(\sqrt 3,1) \kappa_{-3}(M) \right)
+ \delta_{k=2}\mu(M) & \text{if } q = 3.
\end{cases} \]

Using these formulas, we can derive a precise elementary characterization of when the Atkin--Lehner signs at $q$ are perfectly equidistributed.

\begin{prop} \label{prop:req1}
Let $q$ be a prime, $M \ge 1$ be coprime to $q$.  Write $M=2^e M'$, where $M'$ is odd.

\begin{enumerate}
\item Suppose $q \ge 5$ and either $k \ge 4$ or $M$ is not squarefree.  Then
$\Delta_k(q,M) = 0$ if and only if 
(i) $M'$ is not cubefree; (ii) ${-q \leg p} = 1$ for some $p \parallel M'$; 
(iii) $16 \mid M$ and $q \equiv 1 \mod 4$; (iv) $32 \mid M$ and $q \equiv 3 \mod 8$;
or (v) $v_2(M) \ne 0, 4$ and $q \equiv 7 \mod 8$.

\item Suppose $q \ge 5$, $k=2$ and $M$ is squarefree.
Then $\Delta_k(q,M) = 0$ if and only if (i) $M = 1$ and
$q \in \{ 5, 7, 13, 17 \}$; or (ii) $M = 2$ and $q \in \{ 5, 11, 13, 19, 37, 43, 67, 163 \}$.

\item Suppose $q=2$.  We have $\Delta_k(2,M) = 0$ if and only if 
(i) $M$ is not cubefree; (ii) $\prod_{p^2 \parallel M} {2 \leg p}$ is $-1$
if $k \equiv 0, 2 \mod 8$ and $+1$ if $k \equiv 4, 6 \mod 8$; or
(iii) $k=2$ and $M = 1$ or $M \equiv 3, 5 \mod 8$ is prime.

\item Suppose $q=3$.  If $k \ge 4$ or $M$ is not squarefree, then  $\Delta_k(3,M) = 0$ if and only if (i) $M'$ is not cubefree; (ii) ${-3 \leg p} = 1$ for some $p \parallel M'$; (iii) $32 \mid M$; (iv) $M$ is odd and $k \equiv 4, 10 \mod 12$; or (v) $4 \parallel M$ and $k \not \equiv 4, 10 \mod 12$.
 If $k=2$ and $M$ is squarefree, then $\dim S_k^\new(3M)^{+_3} = \dim S_k^\new(3M)^{-_3}$ if and only if $M = 1, 2$. 
\end{enumerate}
\end{prop}

\begin{proof} 

Case (1) follows immediately from the above vanishing conditions for $\alpha_1(-q;e)$.

Case (2): Now suppose $q \ge 5$, $k=2$ and $M$ is squarefree.  Then $\Delta_k(q,M) = 0$ if and only if $\frac 12 \alpha_1(-q;e) \kappa_{-q}(M') = \mu(M)$.  

For a discriminant $\Delta < -4$, $H(\Delta) \ge 1$.
Also if $\Delta = \lambda^2 \Delta_0$ where $\Delta_0$ is a fundamental discriminant and $\lambda > 1$, then $H(\Delta) \ge 2 H(\Delta_0)$.  Thus the only integers $\Delta$ such that $H(\Delta) = 1$ are $\Delta = \{ -7, -8, -11, -19, -43, -67, -163 \}$.  Similarly, there are 19 discriminants $\Delta < 0$ with $H(\Delta) = 2$ (the minimal one is $\Delta = -427$).

If $\kappa_{-q}(M') \ne 0$, it equals $(-2)^{\omega(M')}$.  Thus we can only have
$\Delta_k(q,M) = 0$ if (i) $M' = 1$ and $\alpha_1(-q; e) = \pm 2$, or (ii)  $M' = p$ is prime and $\alpha_1(-q; e) = \pm 1$.  

When $M=1$, $\Delta_k(q,M) = 0$ if and only if $H(-4q) = 2$, i.e., if $q \in \{ 5, 7, 13, 17 \}$.  When $M = 2$, $\Delta_k(q,M) = 0$ if and only if $H(-4q) = 2 H(-q) + 2$, i.e., if and only if $q \in \{ 5, 11, 13, 19, 37, 43, 67, 163 \}$.
When $M = p \ge 3$, $\Delta_k(q,M) = 0$ if and only if $H(-4q) = 1$ and ${-q \leg p} = -1$, which never happens.
When $M = 2p$, for a prime $p \ge 3$, $\Delta_k(q,M) = 0$ if and only if $2H(-q) = H(-4q) + 1$, which also never happens.  This finishes case (2).

Case (3): Next suppose $q=2$, and $k \ne 2$ or $M$ is not squarefree.  Then
$t^\new(1;M) = 0$ if and only if $\kappa_{-2}(M) = \kappa_{-1}(M) = 0$
or $(-1)^{k/2} p_k(\sqrt 2,1) = \prod_{p^2 \parallel M} {2 \leg p}$.  The former never happens.  The latter condition means $\prod_{p^2 \parallel M} {2 \leg p}$ is $-1$ if $k \equiv 0, 2 \mod 8$ and $+1$ if $k \equiv 4, 6 \mod 8$.

If $q=2$, $k=2$ and $M$ is squarefree, then one needs $\kappa_{-2}(M) + \kappa_{-1}(M) = 2\mu(M)$, which is true when $M = 1$ or $M$ is a prime $p \equiv 3, 5 \mod 8$.  This proves case (3).

Case (4): Finally suppose $q=3$.  Then $\Delta_k(q,M) = c \kappa_{-3}(M') + \delta_{k=2} \mu(M)$, where $c = \frac 12(-1)^{k/2} \alpha_1(-3;e) - \frac 13 p_k(\sqrt 3,1) \kappa_{-3}(2^e)$.  One checks that $c \in \{ -1, 0, 1 \}$, and
$c = 0$ if and only if (i) $e \ge 5$; (ii) $e =  0$ and $k \equiv 4, 10 \mod 12$; or (iii) $e = 2$ and $k \equiv 0, 2, 6, 8 \mod 12$.

Hence if $k \ne 2$ or $M$ is not squarefree, then $\Delta_k(q,M) = 0$ if and only if one of (i)--(iii) above holds; (iv) $M'$ is not cubefree; or (v) ${-3 \leg p} = 1$ for some odd $p \parallel M$.

Now suppose $k=2$ and $M$ is squarefree.  Then $c = (-1)^{e+1}$, and we see $\mu(M) = c \kappa_{-3}(M')$ if and only if $M = 1, 2$.  This completes case (4).
\end{proof}

When the Atkin--Lehner signs are not perfectly equidistributed, it is also easy to give conditions for which Atkin--Lehner eigenspace is larger, and give bounds on the differences of dimensions.  For simplicity, we only explicitly do the former when $q \ge 5$.

\begin{prop}
Let $q \ge 5$ be a prime, $M \ge 1$ be coprime to $q$, and $e = v_2(M)$.
Put $\tilde e = 0$ if (i) $e = 0$, (ii) $e = 3$ and $q \equiv 1 \mod 4$, or (iii) $e \ge 4$ and $q \equiv 3 \mod 8$.  Let $\tilde e = 1$ otherwise.
Then
\[ \begin{cases}
\Delta_k(q,M) \ge 0 & \text{if } 
\frac k2 + \omega_1(M) + \omega_2(-q; M) + \tilde e \equiv 0 \mod 2, \\
\Delta_k(q,M) \le 0 & \text{else}.
\end{cases} \]
\end{prop}

\begin{proof}
The sign of $\Delta_k(q,M)$ agrees with the sign of 
$(-1)^{\frac k2} \alpha_1(-q; e) \kappa_{-q}(M')$.  This is immediate from the above expression for $\Delta_k(q,M)$ unless $k = 2$ and $M$ is squarefree.  In that situation it follows as
$\lvert \alpha_1(-q; e) \kappa_{-q}(M') \rvert \ge \lvert \mu(M) \rvert = 1$.
\end{proof}

In particular, the above two propositions contain the $r=1$ case of \cref{thm11}.

\subsection{Dimensions for $r \ge 3$ odd}
Suppose $r \ge 3$ is odd.  Then
\[ \Delta_k(q^r,M) =
\tilde A_{1,0}(r;M) - \tilde A_{1,0}(r-2;M) - \tilde A_{1,1}(r-2;M) + \delta_{r \ge 5} \tilde A_{1,1}(r-4;M). \]

First assume $q^{r-2} > 4$, i.e., $q > 3$ or $r \ge 5$.  Then
\begin{align*}  \Delta_k(q^r,M)  &=
\frac 12(-1)^{\frac k2} \left( \alpha_1(-q^r; e) \kappa_{-q^r}(M') - 2\alpha_1(-q^{r-2}; e) \kappa_{-q^{r-2}}(M') + \delta_{r \ge 5} \alpha_1(-q^{r-4}; e) \kappa_{-q^{r-4}}(M')
\right) \\
&= \frac 12(-1)^{\frac k2} \kappa_{-q}(M') \left( \alpha_1(-q^r; e)  - 2\alpha_1(-q^{r-2}; e)  + \delta_{r \ge 5} \alpha_1(-q^{r-4}; e) \right).
\end{align*}
Recall $\kappa_{-q}(M') = 0$ if and only if $M'$ is not cubefree or ${-q \leg p} = 1$ for some $p \parallel M'$.

Consider the factor $\aleph = \alpha_1(-q^r; e)  - 2\alpha_1(-q^{r-2}; e) + \delta_{r \ge 5} \alpha_1(-q^{r-4}; e)$.  According to \cref{tab:alpha1}, we can write
\[ \aleph = c_1 \left( H(\Delta_0 q^{r-1}) - 2 H(\Delta_0 q^{r-3}) + \delta_{r \ge 5} H(\Delta_0 q^{r-5}) \right), \]
where $\Delta_0 = -4q$ if $q \equiv 1, 2 \mod 4$ and $\Delta_0 =-q$ if $q \equiv 3 \mod 4$.  Here $c_1 \in \{ 0, \pm 1, \pm 2, 4, -6 \}$ according to the cases in \cref{tab:alpha1} (and ${-q \leg 2}$ when $q \equiv 3 \mod 4$).  In particular, $c_1 = 0$ if and only if
 (a) $e \ge 5$, (b) $e = 4$ and $q \equiv 1 \mod 4$, or (c)
$e = 1,2,3$ and $q \equiv 7 \mod 8$.

Moreover, by \cref{tab1} and \eqref{eq:Hlambda2}, we see that\footnote{In the published version, the factor of 2 mistakenly appeared inside the arguments of $\eta_{\Delta_0}$ and $\sigma$.  This final simplified expression for $\aleph$ did not appear in the published version.}
\begin{align*}
 \aleph &= c_1  \left( \eta_{\Delta_0}(q^{\frac{r-1}2}) - 2\eta_{\Delta_0}(q^{\frac{r-3}2})
+ \delta_{r \ge 5}\eta_{\Delta_0}( q^{\frac{r-5}2} ) \right) h'(\Delta_0) \\
&= c_1  \left( \sigma(q^{\frac{r-1}2}) - 2\sigma(q^{\frac{r-3}2})
+ \delta_{r \ge 5}\sigma ( q^{\frac{r-5}2} ) \right) h'(\Delta_0) \\
&= c_1 \left(q^{\frac{r-1}2}  - q^{\frac{r-3}2} \right) h'(\Delta_0).
\end{align*}
In particular, one deduces that $\aleph = 0$ if and only if $c_1 = 0$.
This proves the following.

\begin{prop}
Let $q$ be a prime, $r \ge 3$ odd, $M \ge 1$ be coprime to $q$, and write $M=2^e M'$, where $M'$ is odd.  Assume $q^r \ne 8, 27$.  Then
$\Delta_k(q^r,M) = 0$ if and only if 
(i) $M'$ is not cubefree; (ii) ${-q \leg p} = 1$ for some $p \parallel M'$; (iii) $e \ge 5$;
(iv) $e = 4$ and $q \equiv 1 \mod 4$; or
(v) $e = 1, 2, 3$ and $q \equiv 7 \mod 8$.

Further, when $\Delta_k(q^r,M) \ne 0$, its sign is the sign of $(-1)^{k/2} \alpha_1(-q^r; e) \kappa_{-q^r}(M')$.
\end{prop}

This completes the proof of \cref{thm11}.

Explicit equidistribution criteria for $q^r = 8$ and $q^r = 27$ break up into more cases based on the weight and, in the case of $q^r = 8$, a more delicate relation among quadratic residue symbols.  We merely write down $t^\new(r;M)$ in these cases:

When $q^r = 8$, we have
\[  \Delta_k(8,M)  = \frac 12 \left( (-1)^{\frac k2} \kappa_{-2}(M) + p_k(\sqrt 2, 1) \kappa_{-1}(M) \right). \]

When $q^r = 27$, we have
\[ \Delta_k(27,M)  = \left( \frac 12 (-1)^{\frac k2} (\alpha_1(-27;e) - 2\alpha_1(-3;e)) + \frac 13 p_k(\sqrt 3,1) \kappa_{-3}(2^e) \right) \kappa_{-3}(M'). \]

\subsection{Dimensions for $r = 2$}
Suppose $r=2$.  Then
\[  \Delta_k(q^2,M)  = \tilde A_{1,0}(2;M) - \tilde A_{1,0}(0;M) - \tilde A_{1,1}(0;M) + \tilde A_2(2;M) - \tilde A_2(0;M).  \]
 There does not appear to be a clean characterization of when the Atkin--Lehner sign at $q$ is perfectly equidistributed for general $q, M, k$, but we can give asymptotics for $t^\new(2;M)$ in various parameters.
Let
\[ b_{2,e} = \begin{cases}
1 & \text{if }e = 0, 3, \\
-1 & \text{if }e = 1, 2, \\
0 & \text{if }e \ge 4.
\end{cases} \]

\begin{prop} \label{prop:q2}
\begin{enumerate}
\item Fix $q$. Let $k, M$ denote varying integers such that $k \ge 2$ is even and $M \ge 1$ is coprime to $q$.  As $k+M \to \infty$, we have
\[ \Delta_k(q^2,M) \sim \frac{1}{12} (1-k)\kappa_\infty(M), \]
so in particular $\Delta_k(q^2,M) \to -\infty$.

\item Fix $k, M$.  Suppose $M$ or $\frac M2$ is a cubefree integer such that
$p \not \equiv 1 \mod 4$ for each $p \parallel M$.  Then as $q \to \infty$ along a sequence of primes not dividing $M$, we have
\[ \Delta_k(q^2,M) \sim \frac 14 (-1)^{k/2} b_{2,e} \kappa_{-1}(M') q. \]
\end{enumerate}
In particular, for large $q$, the sign of $\Delta_k(q^2,M)$ is $(-1)^{\frac k2 + b_{2,e} + \omega_1(M') + \omega_2(-1;M')}$.  
\end{prop}

The first part of this proposition coincides with \cref{prop13}.

\begin {proof}
Recall that $\tilde A_2(r;M) = 0$ unless $M$ is a square and $q^r \in \{ 1, 4 \}$, in which case it is $- \frac 12 \alpha_2(M)$.  Note that $0 \le \alpha_2(M) < \sqrt M$.  This combined with our analysis below will imply that the
$\tilde A_2$ terms in $\Delta_k(q^2,M)$ will not contribute to the main asymptotics in either case.

Case (1): Note that for an integer $\Delta$, we have $\lvert \kappa_\Delta(M) \rvert \le 2^{\omega_1(M)}$.
Thus the sum $\tilde A_{1,0}(0;M) + \tilde A_{1,1}(0;M)$ equals
$\frac{2^{\delta_{q=2}}}{12} (k-1)\kappa_\infty(M)$ plus terms that (in absolute value) are $O(2^{\omega_1(M)})$.  The other terms in $\Delta_k(q^2,M)$ are also $O(2^{\omega_1(M)})$, except that when $q^r = 4$ there is also a $\frac{1}{12} \kappa_\infty(M)$ term that cancels out half of the $\kappa_\infty(M)$ contribution from $\tilde A_{1,0}(0;M) + \tilde A_{1,1}(0;M)$.  Since $\kappa_\infty(M) \ge \prod_{p \mid M} (p-1)$, the asymptotic in (1) follows.

Case (2): 
When $q \ne 2$, we have
\[ \tilde A_{1,0}(2;M) - \tilde A_{1,1}(0;M) = \frac 14 (-1)^{k/2} b_{2,e} (q + 1 - {-1 \leg q}) \kappa_{-1}(M'). \]
The hypothesis guarantees that this is nonzero and grows like the asserted multiple of $q$, whereas all other terms in $\Delta_k(q^2,M)$ are bounded independent of $q$.
\end{proof}

\subsection{Dimensions for $r \ge 4$ even}
Now suppose $r \ge 4$ is even.  Then
\[ \Delta_k(q^r,M) =  \tilde A_{1,0}(r;M) - \tilde A_{1,0}(r-2;M) - \tilde A_{1,1}(r-2;M) 
+ \tilde A_{1,1}(r-4;M) + \frac{\delta_{q^r = 16}}2 \alpha_2(M). \]

When $q^r = 16$, we get
\[ \Delta_k(16,M) = \frac 12 \left( (-1)^{\frac k2} \kappa_{-1}(M) + \alpha_2(M) \right), \]  and the $\alpha_2(M)$ term dominates asymptotically if $M \to \infty$ along a sequence of squares.  If $M$ is not a square, then $\alpha_2(M) = 0$ so $\Delta_k(16,M) = 0$ if and only if $M$ is not cubefree or if $p \equiv 1 \mod 4$ for some $p \parallel M$.

Now assume $q^r \ne 16$.  Then
\[ \Delta_k(q^r,M) = \frac 12 (-1)^{\frac k2} \kappa_{-1}(M') \, \aleph, \]
where
\[ \aleph = \alpha_1(-q^r; e) - 2 \alpha_1(-q^{r-2}; e) + \alpha_1(-q^{r-4}; e). \]
From \cref{tab1,tab:alpha1}, we compute
\[ \aleph =
\begin{cases}
2^{\frac{r-4}2} & \text{if }q = 2, \\
\frac 12 q^{\frac{r-4}2} (q-1)(q-{-1 \leg q}) & \text{if }q \ne 2 \text{ and } e = 0, 3, \\
- \frac 12 q^{\frac{r-4}2} (q-1)(q-{-1 \leg q})& \text{if }q \ne 2 \text{ and } e = 1, 2, \\
0 & \text{if }e \ge 4.
\end{cases} \]
This gives an explicit formula for $t^\new(r;M)$, which proves \cref{thm12}.




\section{Correlation of Fourier coefficients and local signs} \label{sec:trl}


Now we will investigate the correlation of Fourier coefficients with Atkin--Lehner signs.  For simplicity, we will only work with Atkin--Lehner operators at primes $q$ that sharply divide the level.  On the other hand, we will consider not just Atkin--Lehner operators $W_q$ at a single prime $q$, but $W_Q = \prod_{q \mid Q} W_q$ for some squarefree $Q \ge 1$.

Let $Q, \ell, M$ be pairwise coprime positive integers with $Q$ squarefree.  Let $k \ge 2$ be even.  From \cite{SZ}, we have
\[ \tr_{S_k(QM)} T_\ell W_Q = A_{1,0}(1; M) + \delta_{Q=1} A_2(0; M) + A_3 \]
where $A_\star$ is defined as in \cref{sec:tnew} with $Q$ in place of $q$. In particular, replacing $s$ with $\frac sQ$ in the definition of $A_{1,\eps}(r;M)$ we have
\[ A_{1,0}(1; M) = - \frac 12 \sum_{s^2 \le \frac{4 \ell}Q} p_k(s \sqrt Q, \ell)
\sum_{\substack{t \mid M \\ M/t \text{ squarefree}}} H_t(s^2 Q^2 - 4 Q \ell). \]
 We also have
\begin{equation} \label{eq:tnewQM}
\tr_{S^\new_k(QM)} T_\ell W_Q = \tilde A_{1,0}(1; M) + \delta_{Q=1} \tilde A_2(0; M) + \tilde A_3,
\end{equation}
where as before $\tilde A_*(r; M) = \sum_{d \mid M} (\mu * \mu)(d) A_*(r; M/d)$.  Here $\tilde A_3$ is given by \eqref{eq:AA3}.  

\subsection{Traces for small $\ell$} \label{sec:trlsmall}

The traces of $T_\ell W_Q$ are simpler when $\ell$ is small relative to $Q$.  In particular, suppose that $4\ell < Q$ so that only the $s=0$ term contributes to $A_{1,0}(1;M)$.  The analysis in \cref{sec:A1a} also applies if we replace $q$ by $Q$, and one has that
\begin{equation} \label{eq:lsmall}
 \tr_{S_k^\new(Q M)} T_\ell W_Q = - \frac 12 (-\ell)^{\frac k2 - 1} \alpha_1(-Q \ell; e) \kappa_{-Q \ell}(M') + \delta_{k=2}\mu(M) \sigma(\ell).
\end{equation}

We have the following consequences for $Q = q$ and $\ell$ both prime.
As before, write  $M = 2^e M'$ with $M'$ odd.

\begin{prop} Assume $\ell < \frac q4$ is prime.  Suppose either $M'$ is not cubefree or $32 \mid M$.  Then the trace of $T_\ell$ on $S_k^\new(q M)^{\pm_q}$ is independent of the sign $\pm_q$ for $\ell < \frac q4$.  Equivalently, since $\Delta_k(q,M) = 0$, the average of $a_\ell(f)$ over newforms in $S_k^\new(q M)^{+_q}$ is equal to that for $S_k^\new(q M)^{-_q}$ for primes $\ell < \frac q4$.
\end{prop}

For simplicity, now assume $M$ is cubefree.  (Note that
$\tr_{S_k^\new(q M)} T_\ell W_q = \delta_{k=2}\mu(M) \sigma(\ell)$ if $\ell < \frac q4$ and $M'$ is not cubefree.)
Then
\[ \kappa_{-q \ell}(M') = \prod_{p \parallel M'} ( {-q \ell \leg p} - 1) \prod_{p^2 \parallel M'} {-q\ell \leg p} \]
and
\[ \alpha_1(-q \ell ;e) = 
\begin{cases}
H(-4q \ell) & \text{if }e = 0, \\
-H(-4 q \ell) & \text{if }e = 1,2 \text{ and } q \ell \equiv 1 \mod 4,  \\
\left({-q \ell \leg 2} - 1 \right) H(-4 q \ell) & \text{if }e = 1,2 \text{ and } q \ell \equiv 3 \mod 4, 
\end{cases} \]
which implies the following.

\begin{prop} Suppose $M$ is cubefree and $\ell < \frac q4$ is prime.  If either
(i) ${-q \ell \leg p} = 1$  for some $p \parallel M'$ or (ii) $e = 1,2$ and $q \ell \equiv 7 \mod 8$, then
\[ \tr_{S_k^\new(q M)} T_\ell W_q = \delta_{k=2}\mu (M)(\ell + 1). \]
Otherwise
\[ 
\tr_{S_k^\new(q M)} T_\ell W_q = c_{1,e}(q\ell) (-2)^{\omega_1(M')-1} (-\ell)^{\frac k2 - 1}  \prod_{p^2 \parallel M'} {-q\ell \leg p}  H(-4 q \ell) + \delta_{k=2}\mu (M)(\ell + 1),
  \]
 where
 \[ c_{1,e}(q\ell) = 
\begin{cases}
1 & \text{if }e = 0, \\
-1 & \text{if }e = 1,2 \text{ and } q \ell \equiv 1 \mod 4,  \\
-2 & \text{if }e = 1,2 \text{ and } q \ell \equiv 1 \mod 8. 
\end{cases}  \]
\end{prop}

Now we compare the sign of this $\tr_{S_k^\new(q M)} T_\ell W_q$ (or whether it is 0 or not) with the sign of $\Delta_k(q,M)$.
For simplicity, assume $k \ge 4$ or $M$ is not squarefree so that the
$\delta_{k=2} \mu(M) \sigma(\ell)$ term vanishes.  Then from 
\cref{prop:req1}, we have $\Delta_k(q,M) = 0$ if and only if (i) ${-q \leg p} = 1$ for some $p \parallel M'$ or (ii) $M$ is even and $q \equiv 7 \mod 8$.

\begin{cor} Suppose $M$ is cubefree, $\ell < \frac q4$ is prime, and either $k \ge 4$ or $M$ is not squarefree.  If the quantities $\Delta_k(q,M) \ne 0$ and
$\tr_{S_k^\new(q M)} T_\ell W_q \ne 0$ are both nonzero, then their ratio has sign $\prod_{p^2 \parallel M'} {\ell \leg p}$.  In particular, their signs are the same if $M'$ is squarefree. 
\end{cor}

This corollary implies the first part of \cref{thm16}.

From \cite[Proposition 14]{murty-sinha}, we have
\[ \left| \tr_{S_k^\new(qM)} T_\ell \right|  \le \ell^{(k-1)/2} \left( 8 \ell^3 4^{\omega(qM)} + \ell^{3/2} \right). \]
In particular, for fixed $\ell, k, M$, we see that $\tr_{S_k^\new(qM)} T_\ell$ is bounded independent of $q$.
Since 
\[ \tr_{S_k^\new(q M)^{\pm q}} T_\ell= \frac 12 \left( \tr_{S_k^\new(qM)}  T_\ell W_q \pm \tr_{S_k^\new(qM)} T_\ell \right), \] 
for $q$ large $H(-4 q \ell)$ dominates $\tr_{S_k^\new(qM)} T_\ell$.
Thus the sign of $\tr_{S_k^\new(q M)^{\pm q}} T_\ell$ will be $\pm 1$ times the sign of $\tr_{S_k^\new(q M)} T_\ell W_q$ for large $q$ such that the latter trace is nonzero.

This proves the second part of \cref{thm16}.

\begin{rem} One could also prove analogous statements for $T_\ell W_Q$.  The restriction to $Q=q$ was simply because our goal here was to compare $T_\ell W_q$ with $\Delta_k(q,M)$.
\end{rem}

\subsection{Traces for general $\ell$  with $M$ squarefree}

Now, assuming $M$ is squarefree, we give a formula for $\tr_{S_k^\new(QM)} T_\ell W_Q$ which is amenable to computation.

Since we restrict to squarefree levels, we specify certain multiplicative functions $\xi_\Delta^\star$ below only on squarefree integers.  To apply standard Dirichlet convolution, we may view these as multiplicative functions on $\mathbb N$ which are, for instance, 0 on non-squarefree numbers.

For fixed $\Delta$ and squarefree $t$, it is straightforward to check that
\[ H_t(\Delta) = \xi^0_\Delta(t) H(\Delta), \]
where $\xi^0_\Delta$ is a multiplicative function satisfying
\[ \xi^0_\Delta(p) =
\begin{cases}
{\Delta \leg p} & \text{if }p^2 \nmid \Delta, \\
p \frac{H(\Delta/p^2)}{H(\Delta)} & \text{if }p^2 \mid \Delta.
\end{cases} \]
Then for squarefree $M$,
\[ \sum_{t | M} H_t(\Delta) = (1 * \xi_\Delta^0)(M) \cdot H(\Delta) =  \xi^1_\Delta(M) H(\Delta), \]
where $\xi^1_\Delta$ is a multiplicative function such that
$\xi^1_\Delta(p) = 1 + \xi^0_\Delta(p)$.

Write
\[ \tilde A_{1,0}(1; M) = - \frac 12 \sum_{s^2 \le \frac{4 \ell} q} p_k(s \sqrt q, \ell) B_1(M, q(s^2 q - 4 \ell)), \]
where \[ B_1(M; \Delta) = \sum_{d \mid M} (-2)^{\omega(d)} \sum_{t \mid M/d}
H_t (\Delta). \]
The above shows that
\[ B_1(M, \Delta) = ((\mu * \mu) * \xi^1_\Delta ) (M) H(\Delta) = \xi_\Delta(M) H(\Delta), \]
where $\xi_\Delta$ is a multiplicative function satisfying $\xi_\Delta(p) = \xi^0_\Delta(p) - 1$.  Using \eqref{eq:Hlambda2}, we can explicitly write
\begin{equation}
\xi_\Delta(p) = \begin{cases}
{\Delta \leg p} -1  & \text{if }p^2 \nmid \Delta, \\
\frac{(p-1)( {\Delta_0 \leg p} - 1)}{(p^{e+1} - 1) - {\Delta_0 \leg p}(p^{e} - 1)}  & \text{if }p^2 \mid \Delta, \, 2e = v_p(\Delta/\Delta_0),
\end{cases}
\end{equation}
where $\Delta_0$ is the negative fundamental discriminant dividing $\Delta$.

When $Q=1$ and $\ell$ prime, we have
\[ A_2(0;N) = - \sum_{t | N} (Q(t), \ell - 1) = - \sigma_0(N). \]
This yields
\[ \tilde A_2(0;N) = -\delta_{N=1}. \]

In summary, we have the following.

\begin{prop} \label{prop:Tl-large}
Let $Q, M, \ell$ be pairwise coprime integers such $N = QM$ is squarefree.  If $Q = 1$, further assume that $\ell$ is prime.  Then
\begin{multline*} \tr_{S_k^\new(QM)} T_\ell W_Q = 
- \frac {\ell^{\frac k2 - 1}}2 \sum_{s^2 \le \frac{4 \ell} Q}
U_{k-2} (\frac s2 \sqrt{\frac Q \ell}) 
\xi_{s^2 Q^2 - 4 Q \ell}(M) H(s^2 Q^2 - 4 Q \ell) \\
 - \delta_{N = 1} + \delta_{k=2}\mu(M) \sigma(\ell).
\end{multline*}
\end{prop}

\begin{rem} When $Q = 1$, this is the squarefree case of the trace formula for $T_\ell$ in \cite{murty-sinha}.  When $M = 1$, this is the trace formula used in \cite{zubrilina}.  Assaf \cite{assaf} also gives a trace formula for $T_\ell W_Q$ on the newspace (without a squarefree level assumption) which involves multiple summations.  The point of \cref{prop:Tl-large} is to give the trace as an explicit linear combination of a minimal collection of class numbers.
\end{rem}


\section{Murmurations} \label{sec:murm}

\subsection{Analysis for Type I} \label{sec:murmI}

Let us now investigate murmurations for arithmetically compatible sequences $(\mathcal N, \mathcal Q) = \{ (N, Q) \}$ of Type I. For simplicity, we will assume $N = QM$ is squarefree where $M$ is fixed and $Q$ ranges over squarefree numbers coprime to $M$.  Fix $k$ and let $\mathcal F$ be the family of weight $k$ newforms of a squarefree level $N \in \mathcal N$.
Then
\begin{equation} \label{eq:AQF-type1}
 A^{\mathcal Q}_{\mathcal F}(\ell,X; \beta) = \frac {\sqrt M}{\#\mathcal F(X,\beta X)}  \ell^{1 - \frac k2} \sideset{}{'}\sum_{\frac{X}M \le Q \le \beta \frac XM}   \tr_{S^\new_k(QM)} T_\ell W_Q,
 \end{equation}
where the prime on the sum means $Q$ is restricted to squarefree numbers coprime to $\ell M$.  Here $\beta > 1$ is fixed.

We want to consider the limit of these averages as $\ell, X \to \infty$ such that $\frac \ell X \to x$ for some $x \in [0,\infty)$ by substituting the trace formula from \cref{prop:Tl-large} into \eqref{eq:AQF-type1}.  Specifically, \cref{conj:WQ-murm} asserts that the limit exists.  Note that one only gets $s$-terms appearing for $s^2 \le \frac{4\ell}Q \le \frac{4 \ell M}X \sim 4M x$.  Consequently, we will only see a bounded number of $s$-terms.  
As $\#\mathcal F(X,\beta X)$ grows like a multiple of $X^2$ (see \cite[Section 3.4]{zubrilina} for a precise estimate), the $\delta_{N=1}$ term will contribute nothing asymptotically, and it is easy to see that the $\delta_{k=2}$ term will asymptotically contribute a constant as $\frac \ell X \to x$.  

Hence it suffices to consider the finitely many $s$-terms from \cref{prop:Tl-large}.  The contribution from each of these terms is determined in \cite{zubrilina} in when $M=1$.  Here we content ourselves with the more modest goal of analyzing the contribution from the $s=0$ term, with the expectation that the work in \cite{zubrilina} can be similarly modified to prove \cref{conj:WQ-murm} in this setting (as well as the variant where $Q$ is restricted to $\mathbb N^\sqf_r$).

The $s=0$ contribution to \eqref{eq:AQF-type1} is
\begin{equation} \label{eq:AQF1-s0}
\frac{(-1)^{\frac k2}}2 \cdot \frac {\sqrt M}{\#\mathcal F(X,\beta X)}  \sideset{}{'}\sum_{\frac{X}M \le Q \le \beta \frac XM} \xi_{-4Q \ell}(M) H(-4Q \ell).
\end{equation}
Since $\xi_{-4Q \ell}(M) = \xi_{-4Q\ell}(2^{v_2(M)}) \prod_{\text{odd }p \mid M} \left( {-Q \ell \leg p} - 1 \right)$, this value only depends on $Q \ell \mod 8M$.

\begin{lem} Fix integers $a, m \ge 1$ such that $(a,m)$ is squarefree.  Then there exists $c > 0$ such that 
\[ \sideset{}{'}\sum_{\substack{1 \le Q < X \\ Q\ell \equiv a \mod m}} H(-4 Q \ell) = c X \sqrt{\ell X} + O(X^{\frac{13}{10}+\eps} + \sqrt \ell \log \ell ), \]
uniformly in $\ell, X$.
\end{lem}

\begin{proof}
We may assume $Q \ell > 3$.  Then $H(-4 Q\ell) = h(-4Q \ell) + h(-Q\ell)$.  Thus the lemma amounts to estimating class number sums in congruence classes with a squarefree restriction on $Q$.  This is similar to classical class number averages, but now with a congruence condition.  Lavrik \cite{lavrik} has already determined class number moments over arithmetic progressions.  The sums of $h(-4Q\ell)$ and $h(-Q\ell)$ are similar; we just explain the proof for $h(-Q\ell)$.

For a discriminant $-D < -4$, we have $h(-D) = \frac{\sqrt D}{\pi} L(1,\chi_{-D})$, where $\chi_{-D}(n) = {-D \leg n}$.  
We have $L(1,\chi_{-D}) = \sum_{n < T} \frac{\chi_{-D}(n)}{n} + O(\sqrt D \log(D)/T)$ by Polya--Vinogradov and partial summation.  Then 
\[ \sideset{}{^*}\sum_{Q < X} L(1, \chi_{-Q\ell}) = \sum_{n^2 < X}\frac 1{n^2}  \sideset{}{^*}\sum_{Q < X} 1 + \sum_{\substack{n < X \\ n \not \in \square}} \frac 1 n \sideset{}{^*}\sum_{Q < X}  {\chi_{-Q\ell}(n)}  + O(\sqrt {\ell / X} \log (\ell X)), \]
where the sums over $Q$ are restricted to squarefree $Q$ coprime to $\ell$ such that $Q\ell \equiv a \mod m$ and $Q \ell \equiv 1 \mod 4$.  

The first double sum on the right is $c_1 X + O(\sqrt X)$ for some fixed $c_1 > 0$.
For the second sum, note that $\left| \sum^*_{Q < X} \chi_{-Q \ell}(n) \right | = \left| {-\ell \leg n} \sum {Q \leg n} \mu^2(Q) \right|$.  By \cite[Lemma 6.7]{zubrilina}, this is $O(Q^{3/5 + \eps} n^{1/5+\eps})$, which implies the second sum above is $O(X^{4/5 + \eps})$.  
This suffices for the corresponding class number sum estimate.
\end{proof}

By the lemma and remarks above, the $s=0$ contribution to \eqref{eq:AQF1-s0} is asymptotic to a finite linear combination of the form $\sum_{i \in \Z/8M \Z} c_i \sqrt{\ell/X}$.  This proves \cref{thm:WQ-murm}, as the hypothesis $x < \frac 1{4M} - \eps$ means that only the $s=0$ term from the trace formula contributes to \eqref{eq:AQF-type1} asymptotically.  

\subsection{Analysis for Type II} \label{sec:murmII}
Next suppose $(\mathcal N, \mathcal Q) = \{ (N, Q) \}$ is of Type II with $Q$ fixed and squarefree, and $N$ ranges over all squarefree numbers of the form $N = QM$.  As before, fix $k$ and let $\mathcal F$ be the family of weight $k$ newforms of some level $N \in \mathcal N$.
Then 
\begin{equation} \label{eq:AQF-type2}
A^{\mathcal Q}_{\mathcal F}(\ell,X; \beta) = \frac 1{\#\mathcal F(X,\beta X)} \ell^{1 - \frac k2} \sideset{}{'}\sum_{\frac XQ \le M \le \beta \frac XQ} 
\sqrt{M} \tr_{S^\new_k(QM)} T_\ell W_Q,
\end{equation}
where the prime on the sum means $M$ is restricted to squarefree numbers coprime to $\ell Q$.

Now we will analyze an analogue of \eqref{eq:AQF-type2} without weighting by the factor $\sqrt M$, and this will motivate its inclusion.

\begin{lem} \label{lem:AQF-untwd}
Let $Q \ge 1$ be squarefree.  As $\ell, X \to \infty$ with $\ell$ prime coprime $Q$, we have
\begin{equation} \label{eq:AQF-unwtd}
\ell^{\frac{1-k}2} \left| \sideset{}{'}\sum_{X < M < \beta X} \tr_{S_k^\new(QM)} T_\ell W_Q  \right| \ll  \ell^{\frac 65 + \eps} X^{\frac 35 + \eps}  + \delta_{k = 2} \, o(\ell X),
\end{equation}
where the sum is restricted to squarefree $M$ such that $(M, Q \ell) = 1$.
\end{lem}

\begin{proof} Note that the $\delta_{k=2}$ term for $\tr_{S_k^\new(QM)} W_Q T_\ell$ contributes $(\ell + 1)o(X)$ to the above sum of traces of $T_\ell W_Q$, using the fact the Mertens function $M(x) = \sum_{n \le x} \mu(n)$ is $o(X)$.  The $\delta_{N=1}$ term in $\tr_{S_k^\new(QM)} W_Q T_\ell$ can be ignored.

For a given $s$ such that $s^2 \le \frac{4 \ell}Q$, the contribution to the above sum is
\[ -\frac 12 U_{k-2} (\frac s2 \sqrt{\frac Q \ell})  H(\Delta) \sideset{}{'}\sum_{X < M < \beta X} \xi_\Delta(M), \]
where $\Delta = s^2 Q - 4 Q \ell$.  Note that $U_{k-2}(\frac s2 \sqrt{\frac Q \ell})$ is absolutely bounded since $U_{k-2}$ is a polynomial and the argument is absolutely bounded.  Since $\lvert \Delta \rvert = O(\ell)$, we have $H(\Delta) = O(\sqrt \ell \log \ell)$, and by \cite[Lemma 6.7]{zubrilina} the sum over $M$ is $O(X^{3/5 + \eps} \ell^{1/5 + \eps})$.  Summing up the $O(\sqrt \ell)$ terms now gives the asserted bound.
\end{proof}

\begin{cor} \label{cor:63}
Let $\mathcal F$ be the family of weight $k$ newforms with squarefree level, and fix $\beta > 1$.  As $\ell, X \to \infty$ such that $\frac \ell X \to x$ for some $x \in [0, \infty)$, the unweighted averages satisfy
\[ A_{\mathcal F}^+(\ell,X; \beta) + A_{\mathcal F}^-(\ell,X; \beta) \to 0. \]
\end{cor}

We actually expect more cancellation than in \cref{lem:AQF-untwd}.
If $\ell, X \to \infty$ such that $\frac \ell X \to x$, then left hand side of \eqref{eq:AQF-unwtd} divided by the number of weight $k$ newforms in that range appears to grow roughly like $\frac{\sqrt \ell}X$.  This suggests the $\sqrt M = \sqrt{N/Q}$ weighting in \eqref{eq:AQF-type2}.

\subsection{Analysis for Atkin--Lehner eigenspaces} \label{sec:eps-murm}
Now fix $m < r$, primes $p_1 <  \dots < p_m$, and let $\mathcal F, \mathcal N$ be as in  \cref{conj:eps-murm}.  Say $N = p_1 \dots p_r \in \mathcal N$ with $p_1 < \dots < p_r$, and let $\eps = (\eps_1, \dots, \eps_r) \in {\pm 1}^r$.  We view $\eps$ as the multiplicative function on divisors $N$ such that $\eps(p_i) = \eps_i$.  Then 
\[ \tr_{S_k^{\new}(N)^\eps} T_\ell = 2^{-r} \sum_{Q \mid N} \eps(Q) \tr_{S_k^\new(N)} T_\ell W_Q. \] 
(See \cite[Proposition 3.2]{me:refdim} for the case of $\ell = 1$, but the proof works for general $\ell$.)

For a subset $I \subset \{ 1, \dots, r \}$, denote by $Q_I = \prod_{i \in I} p_i$ the divisor of some $N \in \mathcal N$, and let $\mathcal Q_I$ the sequence of $Q_I$'s as $N$ ranges over $\mathcal N$.
Since $\dim S_k^{\new}(N)^\eps \approx 2^{-r} \dim S_k^{\new}(N) + O(1)$ by \cite[Corollary 3.4]{me:refdim}, one can approximate the averages for the Atkin--Lehner eigenspace
\[ A^\eps_{\mathcal F}(\ell,X; \beta) \approx 2^{-r} \sum_{I \subset \{ 1, \dots, r \} } \eps(Q_I) c_{I,X,\beta} A^{\mathcal Q_I}_{\mathcal F}(\ell,X; \beta), \quad \text{where } c_{I,X,\beta} =  \sideset{}{'}\sum_{X \le N \le \beta X}  \delta_{N \in \mathcal N} \sqrt{\frac{Q_I}N}. \]

Assuming \cref{conj:WQ-murm}, the terms on the right should only contribute in a limit if $c_{I,X} \not \to 0$ as $X \to \infty$.  Hence we expect 
a relation between the murmurations in \cref{conj:WQ-murm,conj:eps-murm} of the form
\begin{equation}
M^\eps_{\mathcal F}(x; \beta) = \sum_{\{1, \dots, m \} \subset I \subset \{ 1, \dots, r \} } \tilde c_{I,x,\beta} M_{\mathcal F}^{\mathcal Q_I}(x; \beta).
\end{equation}
This justifies the expectations in \cref{rem:murm}(2).


\subsection{Levels divisible by $\ell$} \label{sec:lmidN}
Here we indicate what happens if one includes levels $\ell \mid N$ (as is done in \cite{zubrilina}) in murmurations sums.  For simplicity, let us consider the averages $A^{\pm}_{\mathcal F}(\ell, X; \beta)$ introduced first in \cref{sec:intro-murm}.  For a form $f$ with level $N$ divisible by $\ell$, we have $\lvert \ell^{1-k/2} a_\ell(f) \rvert \le 1$.  Hence
\[ \sum_{\mathcal F^{\pm}(X,\beta X)} \ell^{1-k/2} a_\ell(f)
= \sum_{\mathcal F^{\pm}(X,\beta X)^{(\ell)}} \ell^{1-k/2} a_\ell(f) + O(\frac X \ell). \]
Assuming $\frac \ell X \to x$, the error term in this expression is $O(1)$ and will go to $0$ upon dividing by $\# \mathcal F^{\pm}(X,\beta X)^{(\ell)}$ or $\#\mathcal F^{\pm}(X,\beta X)$.  Since $\# \mathcal F^{\pm}(X,\beta X)^{(\ell)} \approx (1 - \frac 1 \ell) \#\mathcal F^{\pm}(X,\beta X)$, we see there is no asymptotic difference between working with averages over $ \mathcal F^{\pm}(X,\beta X)^{(\ell)}$ or $\mathcal F^{\pm}(X,\beta X)$.  


\section{Quadratic twists} \label{sec:quadtw}

Let $f \in S_k(N)$, and $\chi$ be a quadratic Dirichlet character of conductor $M$.  
From \cite[Proposition 3.1]{atkin-li}, one knows that $f \otimes \chi \in S_k(\lcm(N,M^2))$.
In particular, twisting by $\chi$ acts on eigenforms in $S_k(N)$ if $M^2 \mid N$.  
Note that if $v_p(N) > 2 v_p(M)$ for all $p \mid M$, then twisting by $\chi$ acts on newforms in $S_k(N)$: if $f \in S_k(N)$ is a newform, and $g = f \otimes \chi$ had smaller level $N'$, then necessarily $v_p(N') < v_p(N)$ for some $p \mid M$,
but then $f = g \otimes \chi$ would have level which is strictly smaller than $N$ at $p$.

Here we will examine when twisting by a quadratic character produces a bijection between newforms in $S^\new_k(N)^{+_q}$ and newforms in $S_k^\new(N)^{-_q}$.  
For simplicity we will restrict to the case that $v_p(N) > 2 v_p(M)$ for all $p \mid M$, which is generically necessary.  (If this is not satisfied, there will be some non-minimal forms where twisting by $\chi$ strictly lowers the level, except in the small parameter cases where all relevant lower level spaces are 0-dimensional.)

Say $\pi_q$ is the irreducible admissible representation of $\PGL_2(\Q_q)$ associated to a newform $f \in S_k(q^r M)$, where $(M,q) = 1$ and $r \ge 1$.  Then $r$ is the conductor of $\pi_q$.  If $\pi_q$ is supercuspidal, there are 3 distinct possibilities: (i) it is dihedrally induced from a ramified quadratic extension $E_q/\Q_q$; (ii)
it is dihedrally induced from the unramified quadratic extension of $\Q_q$;
or (iii) it is not dihedrally induced.  We respectively call these cases: (i) ramified supercuspidal; (ii) unramified supercuspidal; and (iii) exceptional supercuspidal.  The exceptional case only happens when $q=2$.

If $r=1$, then $\pi_q$ is an unramified twist of the Steinberg representation.  If $r=2$, then $\pi_q$ can be a ramified principal series, ramified twist of Steinberg, or unramified supercuspidal.  If $r \ge 3$ is odd, then $\pi_q$ is ramified supercuspidal or exceptional supercuspidal (the latter only happens when $q=2$ and $r=3, 7$). If $r \ge 4$ is even, either $\pi_q$ is a ramified principal series representation, unramified supercuspidal, or exceptional supercuspidal (the latter only occurs when $q=2$ and $r = 4, 6$).

For a quadratic Dirichlet character $\chi$, denote by $\kappa(\pi_q, \chi)$ the change in the $W_q$-eigenvalue of $f$ upon twisting by $\chi$, i.e., the ratio of the $W_q$-eigenvalues of $f$ and $f \otimes \chi$.  This only depends on $\pi_q$, and the calculation of $\kappa(\pi_q, \chi)$ is given in \cite{pacetti} (see also \cite{AL,atkin-li} for a more classical perspective in special cases).

Since any quadratic $\chi$ is a product of quadratic characters of prime-power conductor, we may reduce to the case of twisting by characters ramified at a single finite prime $p$.
For an odd prime $p$, let $\chi_p$ denote the quadratic character of conductor $p$, which corresponds to the quadratic extension $\Q(\sqrt{p^*})$ where $p^* = {-1 \leg p} p$.   That is, $\chi_p(n) = {p^* \leg n}$.  For $j \in \{ -1, \pm 2 \}$, let $\chi_j$ be the quadratic character associated to $\Q(\sqrt j)$.  Then $\chi_{-1}$ has conductor 4 and $\chi_{\pm 2}$ has conductor 8.  

\subsection{Twisting at $q$}
First we state $\kappa(\pi_q, \chi_q)$ when $q$ is odd.  Since we are interested in the case where twisting by $\chi_q$ acts on the newforms in $S_k(q^r M)$, we may assume the conductor of $\pi_q$ is $r \ge 3$.

If $\pi_q$ is a ramified principal series, then $\kappa(\pi_q, \chi_q) = {-1 \leg q}$.

If $\pi_q$ is an unramified supercuspidal (so $r$ is even), then $\kappa(\pi_q, \chi_q) = -{-1 \leg q}$.

If $\pi_q$ is a ramified supercuspidal (so $r$ is odd) induced from $E_q/\Q_q$, then $\kappa(\pi_q, \chi_q) = \pm 1)$, where the sign is $+1$ if $E_q = \Q_q(\sqrt {q^*})$ and $-1$ if  $E_p = \Q_p(\sqrt {-q^*})$.

Thus for any $q$ odd and $r \ge 3$, twisting by $\chi_q$ never flips the Atkin--Lehner sign of every kind of representation $\pi_q$ of conductor $r$.  In particular, twisting by $\chi_q$ does not force $\Delta_k(q^r, M) = 0$ (at least assuming that $\dim S_k^\new(q^r M)$ is sufficiently large so all possible local representations occur). 

When $q=2$, the situation is similar.  If $\chi \in \{ \chi_{-1}, \chi_{\pm 2} \}$, 
and $r \ge 5$, one may see from the calculations of $\kappa(\pi_q, \chi)$ in \cite[Theorem 4.2]{pacetti} that twisting by $\chi$ will not flip the Atkin--Lehner sign of each kind of representation $\PGL(\Q_2)$ of conductor $r$.

\subsection{Twisting away from $q$}
Next we consider twisting by a quadratic character ramified only at a prime $p \ne q$.

First suppose $p$ is odd and $p \ne q$.  Then $\kappa(\pi_q, \chi_p) = {q \leg p}^{r}$ for any $\pi_q$ of conductor $r$.

Next let $\chi \in \{ \chi_{-1}, \chi_{\pm 2} \}$.  Then for $q$ odd, we have $\kappa(\pi_q, \chi) = \chi(q)^r$ for any $\pi_q$ of conductor $r$.  In particular, if $r$ is odd then $\kappa(\pi_q, \chi_{-1}) = -1$ if $q \equiv 3 \mod 4$ and $\kappa(\pi_q, \chi_{-1}) = -1$ if $q \equiv 5 \mod 8$.

\begin{prop} \label{prop:quadtwist}
Suppose $N = q^r M$, with $r$ odd, $(q,M) = 1$, and one of the following holds:
\begin{enumerate}
\item there exists an odd $p$ such that $p^3 \mid N$ and ${q \leg p} = -1$;
\item $2^5 \mid N$ and $q \equiv 3 \mod 4$; 
\item $2^7 \mid N$ and $q \equiv 5 \mod 8$.
\end{enumerate}
Then $f \mapsto f \otimes \chi$ defines a bijection of newforms in $S^\new_k(N)^{+_q}$ with $S^\new_k(N)^{-_q}$, where we can take $\chi = \chi_p$ in case (1), $\chi = \chi_{-1}$ in case (2), $\chi = \chi_{\pm 2}$ in case (3).
\end{prop}

Note that when the hypotheses of this proposition hold, one also gets that 
$\tr_{S_k^\new(N)} T_\ell W_q = 0$ for $\ell$ such that $\chi(\ell) = 1$.  Moreover, since $f \equiv f \otimes \chi \mod 2$, each newform in $S_k^\new(N)^{+_q}$ is congruent mod 2 to a newform in $S_k^\new(N)^{-_q}$, and vice versa.


\appendix \section{Errata for ``Rank bias for elliptic curves mod $p$'' by Kimball Martin and Thomas Pharis} \label{appendix}

Here we correct a sign error when $k \equiv 0 \mod 4$ in Section 2 of the published article \cite{me:pharis}.  This has no effect on the rest of the paper.

The following corrections should be made to \cite{me:pharis}:

\begin{enumerate} 
\item p.\ 710, bottom (Section 1A): the phrase ``however the
signs for $k \equiv 0 \mod 4$ are opposite to those for $k \equiv 2 \mod 4$'' should be removed.

\item p.\ 717: The conclusion of Proposition 2.2 should read
\[ \left| \tr_{S_k^\new(N)^\pm} T_n \mp \frac 14 n^{\frac{k-2}2} H(4nN) \right| < 
\left( 2^{\omega(N)} (4n)^{\frac k2} + \delta_{k,2} \right)\sigma_1 (n). \]

\item p.\ 717, proof of Proposition 2.2: $p_k(0,n) = (-n)^{(k-2)/2}$,
not $n^{(k-2)/2}$, so (2-2) should read
\begin{equation} \label{eq:SZ}
 \tr_{S_k(N)} T_n W_N = - \frac 12 (-n)^{\frac{k-2}2} H(4nN) + \delta_{k,2} \sigma_1(n).
 \end{equation}
 Corresponding sign changes should be made throughout of 
 proofs of Proposition 2.2 and Corollary 2.3.

\item p.\ 717: The conclusion of Proposition 2.2 should read
\[ \left| \tr_{S_k^\new(N)^\pm} T_n \mp \frac 14 n^{\frac{k-2}2} H(4nN) \right| < 
\left( 2^{\omega(N)} (4n)^{\frac k2} + \delta_{k,2} \right)\sigma_1(n). \]

\item p.\ 718: The conclusion of Corollary 2.3 should read
\[ N^{\frac 12 - \epsilon} \ll  \pm \tr_{S_k^\new(N)^{\pm}} T_n \ll N^{\frac 12} \log N. \]
 
 \item p.\ 718, bottom: the phrase ``when $k \equiv 2 \mod 4$, and approximately like $\mp \sqrt N$
when $k \equiv 0 \mod 4$'' should be removed.

\end{enumerate}


%
%

\end{document}